\documentclass[12pt]{article}
\usepackage{amsmath}
\usepackage{amsthm}
\usepackage{enumerate}
\usepackage{amssymb}

\usepackage{verbatim}
\usepackage{url}
\usepackage[latin1]{inputenc}

\usepackage{caption}
\usepackage{graphicx,color}
\DeclareMathOperator{\Var}{Var}

\def\be{\begin{equation}}
\def\ee{\end{equation}}
\def\bea{\begin{equation*}}
\def\eea{\end{equation*}}
\def\begs{\begin{split}}
\def\ends{\end{split}}

\newtheorem{thm}{Theorem}[section]
\newtheorem{lma}[thm]{Lemma}
\newtheorem{cor}[thm]{Corollary}
\newtheorem{prop}[thm]{Proposition}

\newtheorem{df}[thm]{Definition}

\theoremstyle{remark}
\newtheorem{preremark}[thm]{Remark}
\newtheorem{preex}[thm]{Example}

\newenvironment{remark}{\begin{preremark}}{\qed\end{preremark}}

\numberwithin{equation}{section}

\newcounter{cnstcnt}
\newcommand{\newconstant}{%
\refstepcounter{cnstcnt}%
\ensuremath{C_{\thecnstcnt}}}
\newcommand{\oldconstant}[1]{\ensuremath{C_{\ref{#1}}}}

\begin{document}

\title{Subdiffusive concentration in first-passage percolation}
\author{Michael Damron \thanks{The research of M. D. is supported by NSF grant DMS-0901534.} \\ \small{Indiana University, Bloomington} \\ \small{Princeton University} \and Jack Hanson \\ \small{Indiana University, Bloomington} \\ \small{Princeton University} \and Philippe Sosoe\thanks{The research of P. S. is supported by an NSERC postgraduate fellowship.} \\ \small{Princeton University}}

\maketitle

\abstract{We prove exponential concentration in i.i.d. first-passage percolation in $\mathbb{Z}^d$ for all $d \geq 2$ and general edge-weights $(t_e)$. Precisely, under an exponential moment assumption $(\mathbb{E}e^{\alpha t_e}<\infty$ for some $\alpha>0$) on the edge-weight distribution, we prove the inequality 
\[
\mathbb{P}\left(|T(0,x)-\mathbb{E}T(0,x)| \geq \lambda \sqrt{\frac{\|x\|_1}{\log\|x\|_1}}\right) \leq ce^{-c'\lambda}, ~\|x\|_1 >1
\] 
for the point-to-point passage time $T(0,x)$. Under a weaker assumption $\mathbb{E}t_e^2(\log t_e)_+<\infty$ we show a corresponding inequality for the lower-tail of the distribution of $T(0,x)$. These results extend work of Bena\"im-Rossignol \cite{BR} to general distributions.}

\section{Introduction}

\subsection{The model}

Let $(t_e)_{e \in \mathcal{E}^d}$ be a collection of non-negative random variables indexed by the nearest-neighbor edges $\mathcal{E}^d$ of $\mathbb{Z}^d$. For $x,y \in \mathbb{Z}^d$, define the passage time
\[
T(x,y) = \inf_{\gamma : x \to y} T(\gamma)\ ,
\]
where $T(\gamma) = \sum_{e \in \gamma} t_e$ and the infimum is over all lattice paths $\gamma$ from $x$ to $y$. $T$ defines a pseudo-metric on $\mathbb{Z}^d$ and First-Passage Percolation is the study of the asymptotic properties of $T$. The model was introduced by Hammersley and Welsh \cite{HW} in 1965 and has recently been the object of much rigorous mathematical progress (see \cite{BS, GK, Howard} for recent surveys).

Under minimal assumptions on $(t_e)$, the passage time $T(0,x)$ is asymptotically linear in $x$, but the lower order behavior has resisted precise quantification. If $d=1$, the passage time is a sum of i.i.d. variables, so its fluctuations are diffusive, giving $\chi = 1/2$, where $\chi$ is the (dimension-dependent) conjectured exponent given roughly by $\Var T(0,x) \asymp \|x\|_1^{2\chi}$. For $d=2$, the minimization in the definition of $T$ is expected \cite{HH} to create subdiffusive fluctuations, with a predicted value $\chi = 1/3$. Subdiffusive behavior is expected in higher dimensions as well.

In this paper, we prove an exponential version of subdiffusive fluctuations for $T(0,x)$ under minimal assumptions on the law of $(t_e)$. This result follows up on work done by the authors (extending the work of \cite{BKS, BR}) in \cite{DHS}, in which it was shown that 
\begin{equation}\label{varb}
\Var T(0,x) \leq C \frac{\|x\|_1}{\log \|x\|_1} \text{ for } \|x\|_1>1
\end{equation}
given only that the distribution of $t_e$ has $2+\log$ moments. Our concentration inequalities apply to a nearly optimal class of distributions: for the upper tail inequality in \eqref{eq: main_thm} we require that $t_e$ has exponential moments and for the lower tail inequality \eqref{eq: main_thm_2}, that $t_e$ has $2+\log$ moments. In contrast to existing work on subdiffusive concentration listed below, our methods do not rely on any properties of the distribution other than the tail behavior.

In 1993, Kesten \cite[Eq.~(1.15)]{kestenspeed} gave the first exponential concentration inequality for $T$, showing that if $\mathbb{E}e^{\alpha t_e}<\infty$ for some $\alpha>0$, then one has
\[
\mathbb{P}\left(|T(0,x) - \mathbb{E}T(0,x)| \geq \lambda \sqrt{\|x\|_1} \right) \leq c e^{-c'\lambda} \text{ for } \lambda \leq c'' \|x\|_1\ .
\]
This was improved by Talagrand \cite[Eq.~(8.31)]{talagrand} to Gaussian concentration under the same moment assumption. We provide an alternative derivation of Talagrand's result based on the entropy method in Section \ref{sec: talagrand} of the Appendix.

In the influential work of Benjamini-Kalai-Schramm \cite{BKS}, an inequality of Talagrand for functions of Bernoulli random variables \cite[Theorem 1.5]{talagrand-russo} was used to derive a sublinear bound \eqref{varb} when $\mu$ is supported on two positive values $a$ and $b$. Although the mechanism exploited in \cite{BKS} (small \emph{influences} of the edge variables) did not appear to depend on the edge-weight distribution, the inequality in \cite{talagrand-russo} is specific to the case of Bernoulli variables (Gaussian versions have also been derived, see \cite{bobkov-houdre} and \cite[Theorem 5.1]{chatterjee}). The problem of extending the result to more general distributions was posed in the introduction of \cite{BKS}. There, the authors highlight the connections with hypercontractivity and the Bonami-Beckner inequality, suggesting that these might be useful in generalizing their theorem. An alternative interpretation of the argument of Benjamini-Kalai-Schramm was given in \cite[Theorem VII.1]{garbansteif}. See also \cite[Section 5.5]{chatterjee} for a simple proof of a result analogous to that of Benaim-Rossignol.

Hypercontractivity is well-known to be related to log-Sobolev inequalities \cite{gross}. Bena\"im and Rossignol \cite{BR} developed a ``modified Poincar\'e'' inequality analogous to that of Talagrand for distributions satisfying a log-Sobolev inequality. They applied this to prove a subdiffusive concentration inequality of the type \eqref{eq: main_thm} for a class of distributions that they called \emph{nearly Gamma}. These are continuous distributions that satisfy an entropy bound analogous to the logarithmic Sobolev inequality for the gamma distribution: for nearly gamma $\mu$ and, for simplicity, $f$ smooth,
\[
Ent_\mu~f^2 := \int f^2(x) \log \frac{f^2(x)}{\mathbb{E}_\mu f^2} ~\mu(\text{d}x) \leq C \int (\sqrt x f'(x))^2~\mu(\text{d}x)\ .
\]
Although the nearly Gamma class includes, for example, gamma, uniform and distributions with smooth positive density on an interval, it excludes all power law distributions, those with unbounded support which decay too quickly, those with zeros on its support and all noncontinuous distributions.

The main goal of this paper is to prove subdiffusive concentration inequalities without the nearly Gamma assumption. In a previous paper \cite{DHS}, the authors proved that the variance bound \eqref{varb} holds for general distributions under a weak moment assumption. This shows in particular that the argument of Benjamini, Kalai and Schramm based on small influences of the edge variables does not depend on delicate features of the distribution such as the existence of a log-Sobolev inequality or the nearly Gamma property. The contribution of the current work is to explain  how the ideas in \cite{DHS} can be combined with the ``entropy method'' of Ledoux \cite{ledoux}, as further developed in particular by Boucheron-Lugosi-Massart \cite{BLM}, to yield exponential concentration.

We describe the technical features of the current paper in more detail in Section \ref{sec: technic}. Before ending the current section, let us make a comment regarding the scope of our method. Given the wide applicability of the entropy method, and the presentation in \cite{BR}, one can wonder whether a sublinear variance bound \`a la Talagrand, or a concentration result might hold for general functions of random variables with ``small influences.'' We do not derive such a result. Ours is genuinely a first-passage percolation result, and we use the structure of the model in an essential way in several places, notably in the computation of the linearization of the passage time (Lemma \ref{lem: DHS_lem}), and the lattice animal argument in Section~\ref{sec: animals}.
\subsection{Main result}
The main assumptions are as follows: $\mathbb{P}$, the distribution of $(t_e)$, is a product measure on $\Omega=[0,\infty)^{\mathcal{E}^d}$ with marginal $\mu$ satisfying
\begin{equation}\label{eq: percolation_assumption}
\mu(\{0\})< p_c\ ,
\end{equation}
where $p_c$ is the critical probability for $d$-dimensional bond percolation.  Furthermore, we will generally assume either
\begin{equation}\label{eq: exponential_moments}
\mathbb{E}e^{2\alpha t_e}<\infty \text{ for some } \alpha>0
\end{equation}
or
\begin{equation}\label{eq: two_moments}
\mathbb{E}t_e^2(\log t_e)_+ < \infty\ .
\end{equation}

\begin{thm}\label{thm: main_thm}
Let $d\geq 2$. Under \eqref{eq: percolation_assumption} and \eqref{eq: exponential_moments}, there exist $c_1,c_2>0$ such that for all $x \in \mathbb{Z}^d$ with $\|x\|_1 > 1$,
\begin{equation}\label{eq: main_thm}
\mathbb{P}\left( |T(0,x) - \mathbb{E}T(0,x)| \geq \frac{\|x\|_1^{1/2}}{(\log \|x\|_1)^{1/2}} \lambda \right) \leq c_1e^{-c_2 \lambda} \text{ for } \lambda \geq 0\ .
\end{equation}
Assuming \eqref{eq: percolation_assumption} and \eqref{eq: two_moments}, there exist $c_1,c_2>0$ such that for all $x \in \mathbb{Z}^d$ with $\|x\|_1 > 1$,
\begin{equation}\label{eq: main_thm_2}
\mathbb{P}\left( T(0,x) - \mathbb{E}T(0,x) \leq -\frac{\|x\|_1^{1/2}}{(\log \|x\|_1)^{1/2}} \lambda \right) \leq c_1e^{-c_2 \lambda} \text{ for } \lambda \geq 0\ .
\end{equation}
\end{thm}

\begin{remark}
If \eqref{eq: percolation_assumption} fails, then $T(0,x)$ itself is sublinear in $x$ and the model has a different character (see \cite[Theorem~6.1]{kesten} and \cite{chayes, zhang} for more details). Because $T(0,x)$ is bounded below by the minimum of the $2d$ weights of edges adjacent to 0, it is necessary to assume \eqref{eq: exponential_moments} to obtain an upper-tail exponential concentration inequality. For the lower tail, our methods require $2+\log$ moments and this is the same assumption made in \cite{DHS} for a sublinear variance bound. 
\end{remark}

\subsection{Outline of the proof}
\label{sec: technic}
The strategy of the proof is to use a relation, stated in Lemma~\ref{lem: iteration}, between bounds on $\Var e^{\lambda T(0,x)}$ and exponential concentration. To obtain the required variance bound (Theorem~\ref{thm: F_m_variance_bound}), we apply the Falik-Samorodnisky inequality (Lemma~\ref{lem: lower_bound}) to a martingale decomposition of the variable $e^{\lambda F_m}$, where $F_m$ is an averaged version of the passage time. This approach was first taken by Bena\"im and Rossignol in \cite{BR}. The tails of the true passage time $T(0,x)$ can be estimated from those of $F_m$, although the methods are different for the upper and lower tail (Section \ref{sec: reduction}). 

We represent the passage times as a push-foward of Bernoulli sequences (Section \ref{sec: benc}) and the bound follows after a careful analysis of discrete derivatives resulting from an application of the log-Sobolev inequality for Bernoulli variables. This analysis, presented in Sections 4.2 and 4.3, is the central part of the argument. As in \cite[Proposition 6.4]{DHS}, an important tool in understanding the discrete derivatives is the linearization of the passage time, Lemma~\ref{lem: DHS_lem}. One of the complications we must address is dealing with the function $x\mapsto e^{\lambda x}$  applied to the passage time, rather than the passage time itself. In particular, different arguments are required for the cases $\lambda \ge 0$ and $\lambda < 0$.

Once we have obtained the estimates in Theorems \ref{prop: intermediate_1} and \ref{prop: intermediate} for the entropy sums, it remains to bound the quantities
\[\sum_{e\in Geo(z,x+z)}(1-\log F(t_e)),\]
and decouple them from the exponential terms to apply the iteration idea contained in Lemma \ref{lem: iteration}. Here $F$ is the distribution function of $\mu$.
To accomplish this, we use greedy lattice animals (Section~\ref{sec: animals}). This was one of the new ideas introduced in \cite{DHS}. The treatment here, based on the observation contained in Proposition~\ref{prop:expdist}, is considerably simpler than that in \cite[Section 6.2.3]{DHS}. Even so, the case $\lambda < 0$ requires a separate argument, using an idea from \cite{DK}.

\subsection{Preliminary results}

We will need a couple of results on the length of geodesics. By Proposition~\ref{prop: kesten_length} below, condition \eqref{eq: percolation_assumption} ensures that
\begin{equation}\label{eq: geodesics}
\mathbb{P}(\exists \text{ a geodesic from } x \text{ to } y) = 1 \text{ for all } x,y \in \mathbb{Z}^d\ ,
\end{equation}
where a geodesic is a path $\gamma$ from $x$ to $y$ that has $T(\gamma) = T(x,y)$.

The fundamental estimate is from Kesten \cite[Proposition~5.8]{kesten}.
\begin{prop}[Kesten]\label{prop: kesten_length}
Assuming \eqref{eq: percolation_assumption}, there exist $a,\newconstant\label{C_1}>0$ such that for all $n \in \mathbb{N}$,
\begin{equation}\label{eq: kesten_length}
\mathbb{P}\bigg( \exists \text{ self-avoiding } \gamma \text{ containing }0 \text{ with } \#\gamma \geq n \text{ but } T(\gamma) < an \bigg) \leq e^{-\oldconstant{C_1}n}\ .
\end{equation}
\end{prop}

As a consequence, we state a bound used in work of one of the authors and N. Kubota \cite{DK}. For this, let $G(0,x)$ be the maximal number of edges in any self-avoiding geodesic from $0$ to $x$. An application of Borel-Cantelli to \eqref{eq: kesten_length} implies
\begin{equation}\label{eq: lower_comparability}
\text{under \eqref{eq: percolation_assumption}, }\liminf_{\|x\|_1 \to \infty} \frac{T(0,x)}{\|x\|_1} \geq a > 0 \text{ almost surely}
\end{equation}
and so $G(0,x)$ is finite almost surely. 
\begin{prop}\label{prop: kubota}
Assume \eqref{eq: percolation_assumption} and let $a$ be from Proposition~\ref{prop: kesten_length}. There exists $\newconstant\label{C_2}$ such that the variable
\[
Y_x = G(0,x) \mathbf{1}_{\{T(0,x) < aG(0,x)\}}
\]
satisfies $\mathbb{P}(Y_x \geq n) \leq e^{-\oldconstant{C_2} n}$ for all $x \in \mathbb{Z}^d$ and $n \in \mathbb{N}$.
\end{prop}
\begin{proof}
Defining
\[
A_m = \bigg\{ \exists \text{ self-avoiding } \gamma \text{ from 0 with }\#\gamma \geq m \text{ but } T(\gamma) < a\#\gamma \bigg\}
\]
and summing \eqref{eq: kesten_length} over $n$, one has, for some $\oldconstant{C_2}>0$,
\begin{equation}\label{eq: sum_bound}
\mathbb{P}(A_m) \leq e^{-\oldconstant{C_2}m} \text{ for all } m \in \mathbb{N}\ .
\end{equation}
For $x \in \mathbb{Z}^d$, assume $Y_x \geq n \geq 1$ and let $\gamma$ be any self-avoiding geodesic from $0$ to $x$ with length $G(0,x) = Y_x$. Then because $Y_x \neq 0$, $G(0,x) > (1/a)T(0,x)$ and so 
\[
T(\gamma) = T(0,x) < aG(0,x) = a\#\gamma\ ,
\]
with $\#\gamma \geq n$. So $A_n$ occurs and \eqref{eq: sum_bound} completes the proof.
\end{proof}

We can state a couple of relevant consequences of this proposition. 
\begin{cor}\label{cor: exponential_length_cor}
Assume \eqref{eq: percolation_assumption}.
\begin{enumerate}
\item There exists $\newconstant\label{C_3}$ such that
\begin{equation}\label{eq: two_length_bound}
\mathbb{E}G(0,x)^2 \leq \oldconstant{C_3}\mathbb{E}T(0,x)^2 \text{ for all } x \in \mathbb{Z}^d\ .
\end{equation}
\item Under \eqref{eq: exponential_moments}, there exists $\alpha_1>0$ such that
\begin{equation}\label{eq: exponential_length_bound}
\sup_{0 \neq x \in \mathbb{Z}^d} \frac{\log \mathbb{E} e^{\alpha_1 G(0,x)}}{\|x\|_1} < \infty\ .
\end{equation}
\end{enumerate}
\end{cor}
\begin{proof}
Estimate
\[
\mathbb{E}G(0,x)^2 \leq a^{-2} \mathbb{E} T(0,x)^2 + \mathbb{E}Y_x^2\ ,
\]
where $Y_x$ is from Proposition~\ref{prop: kubota}. Because $\mathbb{E}Y_x^2$ is bounded uniformly in $x$, \eqref{eq: lower_comparability} (which gives $\mathbb{E}T(0,x) \to \infty$ as $\|x\|_1 \to \infty$) shows \eqref{eq: two_length_bound}. Assuming \eqref{eq: exponential_moments}, for $\beta>0$,
\[
\mathbb{E}e^{\beta G(0,x)} \leq \mathbb{E}e^{(\beta/a) T(0,x)} + \mathbb{E}e^{\beta Y_x}\ .
\]
For $\beta < \oldconstant{C_2}/2$, the second term is bounded in $x$. On the other hand letting $\gamma_x$ be a deterministic path from $0$ to $x$ of length $\|x\|_1$, the first term is bounded by
\begin{equation}\label{eq: deterministic_bound}
\mathbb{E}e^{(\beta/a)T(\gamma_x)} = \left( \mathbb{E}e^{(\beta/a) t_e} \right)^{\|x\|_1}\ .
\end{equation}
So we conclude for $x \neq 0$ and some $\newconstant\label{C_4} \geq 1$,
\[
\frac{\log \mathbb{E} e^{\beta G(0,x)}}{\|x\|_1} \leq \frac{\log \left( \mathbb{E}e^{(\beta/a)t_e} + \oldconstant{C_4} \right)^{\|x\|_1}}{\|x\|_1} = \log \left(\mathbb{E}e^{(\beta/a)t_e}+\oldconstant{C_4} \right) < \infty
\]
when $\beta < \min \{\oldconstant{C_2}/2, \alpha a\}$, where $\alpha$ is from \eqref{eq: exponential_moments}.
\end{proof}

\[
\text{{\bf For the remainder of the paper we assume \eqref{eq: percolation_assumption}.}}
\]

\section{Setup for the proof}

Instead of showing concentration for $T(0,x)$, we use an idea from \cite{BKS}: to show it for $T(z,z+x)$, where $z$ is a random vertex near the origin. So, given $x \in \mathbb{Z}^d$, fix $\zeta$ with $0 < \zeta < 1/4$ and define
\[
m = \lfloor \|x\|_1^\zeta \rfloor \text{ and } F_m = \frac{1}{\# B_m} \sum_{z \in B_m} T_z\ ,
\]
where $T_z = T(z,z+x)$ and $B_m = \{y \in \mathbb{Z}^d : \|y\|_1 \leq m\}$ (this particular randomization was used by both \cite{AZ} and \cite{Sodin}). For $\lambda \in \mathbb{R}$ we define
\begin{equation}\label{eq: G_def}
G=G_\lambda = e^{\lambda F_m}\ .
\end{equation}
Below are the concentration inequalities for $F_m$ analogous to \eqref{eq: main_thm}. In the next subsection, we will show why they suffice to prove Theorem~\ref{thm: main_thm}. 
\begin{thm}\label{thm: F_m_concentration}
Assuming \eqref{eq: exponential_moments}, there exist $c_1,c_2>0$ such that for all $x \in \mathbb{Z}^d$ with $\|x\|_1 > 1$,
\[
\mathbb{P}\left( |F_m - \mathbb{E}F_m| \geq \frac{\|x\|_1^{1/2}}{(\log \|x\|_1)^{1/2}} \lambda \right) \leq c_1e^{-c_2 \lambda} \text{ for } \lambda \geq 0\ .
\]
Assuming \eqref{eq: two_moments}, there exist $c_1,c_2>0$ such that for all $x \in \mathbb{Z}^d$ with $\|x\|_1 > 1$,
\[
\mathbb{P}\left( F_m - \mathbb{E}F_m < -\frac{\|x\|_1^{1/2}}{(\log \|x\|_1)^{1/2}} \lambda \right) \leq c_1e^{-c_2 \lambda} \text{ for } \lambda \geq 0\ .
\]
\end{thm}

This theorem is a consequence on a bound for $\Var e^{\lambda F_m}$, and this is what we focus on from Section~\ref{sec: influences} onward. The link between a variance bound and concentration is given by the following lemma from \cite[Lemma~4.1]{BR} (which itself is a version of \cite[Corollary~3.2]{ledoux}). We have split the statement from \cite{BR} into upper and lower deviations.
\begin{lma}\label{lem: iteration}
Let $X$ be a random variable and $K>0$. Suppose that 
\[
\Var e^{\lambda X/2} \leq K\lambda^2 \mathbb{E}e^{\lambda X} < \infty \text{ for } 0 \leq \lambda < \frac{1}{2\sqrt K}\ .
\]
Then 
\[
\mathbb{P}\left( X - \mathbb{E}X > t \sqrt{K}\right) \leq 2e^{-t} \text{ for all } t \geq 0\ .
\]
If 
\[
\Var e^{\lambda X/2} \leq K\lambda^2 \mathbb{E}e^{\lambda X} < \infty \text{ for } -\frac{1}{2\sqrt K} < \lambda \leq 0\ ,
\]
then
\[
\mathbb{P}\left( X - \mathbb{E}X < -t \sqrt{K}\right) \leq 2e^{-t} \text{ for all } t \geq 0\ .
\]
\end{lma}

Taking $K = \frac{C \|x\|_1}{\log \|x\|_1}$ for $\|x\|_1 > 1$ and $X=F_m$ in the previous lemma shows that to prove Theorem~\ref{thm: F_m_concentration}, it suffices to show the following variance bound.
\begin{thm}\label{thm: F_m_variance_bound}
Assuming \eqref{eq: exponential_moments}, there exists $\newconstant\label{D_1}>0$ such that 
\begin{equation}\label{eq: variance_1}
\Var e^{\lambda F_m/2} \leq K \lambda^2 \mathbb{E}e^{\lambda F_m} < \infty \text{ for } |\lambda| < \frac{1}{2\sqrt K} \text{ and }\|x\|_1 > 1\ ,
\end{equation}
where $K = \frac{\oldconstant{D_1} \|x\|_1}{\log \|x\|_1}$. Assuming \eqref{eq: two_moments}, there exists $\newconstant\label{D_2}>0$ such that
\begin{equation}\label{eq: variance_2}
\Var e^{\lambda F_m/2} \leq K \lambda^2 \mathbb{E}e^{\lambda F_m} < \infty \text{ for } -\frac{1}{2\sqrt K} < \lambda \leq 0 \text{ and } \|x\|_1 > 1\ ,
\end{equation}
where $K = \frac{\oldconstant{D_2} \|x\|_1}{\log \|x\|_1}$.
\end{thm}
The proof of this bound will be broken into several sections below.

\subsection{Theorem~\ref{thm: F_m_concentration} implies Theorem~\ref{thm: main_thm}}\label{sec: reduction}

Assume first that we have the concentration bound
\begin{equation}\label{eqn: Fmconc}
\mathbb{P}\left(|F_m-\mathbb{E}F_m| \ge \frac{\|x\|^{1/2}_1}{(\log \|x\|_1)^{1/2}}\lambda\right) \le b e^{-c\lambda},~ \lambda \geq 0
\end{equation}
for some $b,c>0$ and that \eqref{eq: exponential_moments} holds. We will derive from \eqref{eqn: Fmconc} the corresponding estimate for the passage time $T=T(0,x)$:
\begin{equation}\label{eq: T_concentration}
\mathbb{P}\left(|T(0,x)-\mathbb{E}T(0,x)| \ge \frac{\|x\|_1^{1/2}}{(\log \|x\|_1)^{1/2}}\lambda\right) \le b'e^{-c'\lambda},~ \lambda \geq 0\ .
\end{equation}
Write
\[
T(0,x)-\mathbb{E}T(0,x) = F_m-\mathbb{E}T + T(0,x)-F_m\ ,
\]
and note that $\mathbb{E}F_m = \mathbb{E}T$.
If both events $\{|F_m - \mathbb{E}F_m| < \lambda/2\}$ and $\{|T(0,x)-F_m| <\lambda/2\}$ occur, then the triangle inequality implies that we have the bound
\[
|T(0,x)-\mathbb{E}T(0,x)|< \lambda\ .
\]
This results in the estimate
\begin{equation}\label{eqn: triangleineq}
\mathbb{P}(|T(0,x)-\mathbb{E}T(0,x)| \ge \lambda)\le \mathbb{P}(|F_m-\mathbb{E}F_m| \ge \lambda/2)+\mathbb{P}(|T(0,x)-F_m|\ge \lambda/2)\ .
\end{equation}
By subadditivity, we can write
\begin{align*}
|T(0,x) - F_m| = \left| T(0,x)-\frac{1}{\sharp B_m}\sum_{z\in B_m} T(z,z+x)\right| &\le \frac{1}{\sharp B_m}\sum_{z\in B_m}|T(0,x)-T(z,z+x)| \\
&\le \frac{1}{\sharp B_m}\sum_{z\in B_m}(T(0,z)+T(x,x+z))\ .
\end{align*}
Repeating the argument for \eqref{eq: deterministic_bound} (bounding $T(0,z)$ by the passage time of a deterministic path), we have for $\alpha \geq 0$ and each $z\in B_m$
\[ 
\mathbb{E}e^{\alpha T(0,z)} \le \left(\mathbb{E}e^{\alpha t_e}\right)^{\|x\|_1^\zeta} = \mathbf{C}(\alpha)^{\|x\|_1^\zeta}\ .
\]
Here $\alpha$ is from \eqref{eq: exponential_moments}. We now obtain a bound for the second term on the right in \eqref{eqn: triangleineq}. Let $M>0$. First, by a union bound,
\[\mathbb{P}\left(\frac{1}{\sharp B_m}\sum_{z\in B_m}(T(0,z)+T(x,x+z)) \ge 2M\right) \le 2\mathbb{P}\left( \max_{z\in B_m} T(0,z) \ge M\right).\]
The last quantity is bounded by
\begin{align*}
2(\sharp B_m) \cdot\max_{z\in B_m} \mathbb{P}(T(0,z)\ge M) &\le 2(\sharp B_m) \cdot e^{-\alpha M}\mathbf{C}(\alpha)^{\|x\|_1^\zeta}\\
&\le2\|x\|_1^{d\zeta}e^{-\alpha M}\mathbf{C}(\alpha)^{\|x\|_1^\zeta}\ .
\end{align*}
Choosing $2M= \lambda \|x\|^{1/2}_1/(\log\|x\|_1)^{1/2}$ and adjusting constants, we find the bound
\[
\mathbb{P}\left( |T(0,x)-F_m|\ge \lambda \|x\|^{1/2}_1/(\log\|x\|_1)^{1/2} \right) \leq b'e^{-c' \lambda}\ .
\]
Combined with \eqref{eqn: Fmconc} in \eqref{eqn: triangleineq}, this shows \eqref{eq: T_concentration}.

We now move to proving that under assumption \eqref{eq: two_moments}, if we prove Theorem~\ref{thm: F_m_concentration}, then there exist $b'',c''>0$ such that
\[
\mathbb{P}\left( T(0,x) - \mathbb{E}T(0,x) < -\frac{\|x\|_1^{1/2}}{(\log \|x\|_1)^{1/2}}\lambda\right) \leq b''e^{-c'' \lambda},~ \lambda \geq 0\ .
\]
Defining $S = \sum_{e \in B_m} t_e$, where the sum is over all edges with both endpoints in $B_m$, then
\[
\mathbb{P}(S \leq 2\mathbb{E}S) \geq 1/2\ .
\]
By the Harris-FKG inequality \cite[Theorem~2.15]{BLM}, if we put $c_x = \frac{\|x\|_1^{1/2}}{(\log \|x\|_1)^{1/2}}$ and $S' = \sum_{e \in x+B_m} t_e$, then
\[
\mathbb{P}(T(0,x) - \mathbb{E}T(0,x) \leq -\lambda c_x,~S \leq 2\mathbb{E}S,~S' \leq 2\mathbb{E}S') \geq (1/4) \mathbb{P}(T(0,x) - \mathbb{E}T(0,x) \leq -\lambda c_x )\ .
\]
This means that
\[
\mathbb{P}(T(0,x) - \mathbb{E}T(0,x) \leq -\lambda c_x) \leq 4 \mathbb{P}(T(0,x) - \mathbb{E}T(0,x) \leq -\lambda c_x, ~S \leq 2\mathbb{E}S, ~S' \leq 2\mathbb{E}S')\ .
\]
However the event on the right implies that for any $z \in B_m$, $T(z,z+x) \leq T(0,x)+ 4\mathbb{E}S$. Therefore
\[
\mathbb{P}(T(0,x) - \mathbb{E}T(0,x) \leq -\lambda c_x) \leq 4 \mathbb{P}(F_m - \mathbb{E}F_m \leq -\lambda c_x + 4\mathbb{E}S)\ .
\]
Now we can bound $4\mathbb{E}S$ by $\newconstant\label{newie} \|x\|_1^{d\zeta}$, so
\[
\mathbb{P}(T(0,x) - \mathbb{E}T(0,x) \leq -\lambda c_x) \leq 4 \mathbb{P}(F_m - \mathbb{E}F_m \leq -\lambda c_x + \oldconstant{newie}\|x\|_1^{d\zeta})
\]
and this is bounded by
\[
4\mathbb{P}\left( F_m-\mathbb{E}F_m \leq -\left(\lambda - \oldconstant{newie}\frac{\|x\|_1^{d\zeta}}{c_x} \right)c_x \right) \leq 4c_1 \exp\left( - c_2 \left( \lambda - C_7 \frac{\|x\|_1^{d\zeta}}{c_x}\right) \right)\ ,
\]
as long as $\lambda \geq C_7 \frac{\|x\|_1^{d\zeta}}{c_x}$. 

To finish, we simply choose $\zeta = (4d)^{-1}$, so that $\|x\|_1^{d\zeta}/c_x \leq \newconstant\label{newie_2}$ for $\|x\|_1 > 1$ and some $\oldconstant{newie_2}>0$. This  implies
\[
\mathbb{P}(T(0,x) - \mathbb{E}T(0,x) \leq -\lambda c_x) \leq 4 c_1 \exp\left( -c_2 (\lambda - \oldconstant{newie_2}) \right) \text{ for } \lambda \geq 0, ~\|x\|_1>1\ ,
\]
giving the bound $\newconstant e^{-\newconstant \lambda}$.

\subsection{Falik-Samorodnitsky and entropy}

Enumerate the edges of $\mathcal{E}^d$ as $e_1, e_2, \ldots$ and write $e^{\lambda F_m}$ as a sum of a martingale difference sequence:
\[
G-\mathbb{E}G = \sum_{k=1}^\infty V_k\ ,
\]
where
\begin{equation} \label{eqn: Vkdef}
V_k = \mathbb{E}[G \mid \mathcal{F}_k] - \mathbb{E}[G \mid \mathcal{F}_{k-1}]
\end{equation}
and $G$ was defined in \eqref{eq: G_def}. We have written $\mathcal{F}_k$ for the sigma-algebra $\sigma(t_{e_1}, \ldots, t_{e_k})$, with $\mathcal{F}_0$ trivial. In particular if $F\in L^1(\Omega, \mathbb{P})$,
\begin{equation}
\mathbb{E}[F\mid \mathcal{F}_k] = \int F\big((t_e)\big)\, \prod_{i\ge k+1}\mu(\mathrm{d}t_{e_i}).
\end{equation}

To prove concentration for $F_m$, we bound the variance of $G$; the lower bound comes from the proof of \cite[Theorem~2.2]{FS}.
\begin{lma}[Falik-Samorodnitsky]\label{lem: lower_bound}

If $\mathbb{E}G^2<\infty$,
\begin{equation}\label{eq: BR_lower_bound}
\sum_{k=1}^\infty Ent(V_k^2) \geq \Var G ~ \log \left[ \frac{\Var G}{\sum_{k=1}^\infty (\mathbb{E} |V_k|)^2}\right] \ .
\end{equation}
\end{lma}

In the above lemma, we have used $Ent$ to refer to entropy:
\begin{df}
If $X$ is a non-negative random variable with $\mathbb{E}X<\infty$ then the entropy of $X$ is defined by:
\[
Ent~X = \mathbb{E}X\log X - \mathbb{E}X \log \mathbb{E}X\ .
\]
\end{df}

We will need some basic results on entropy. This material is taken from \cite[Section~2]{DHS}, though it appears in various places, including \cite{BLM}. By Jensen's inequality, $Ent~X \geq 0$. There is a variational characterization of entropy \cite[Section~5.2]{ledoux} that we will use.
\begin{prop}\label{prop: variation}
We have the formula
\[
Ent~X = \sup\{\mathbb{E}XY : \mathbb{E}e^Y \leq 1\}\ .
\]
\end{prop}
The second fact we need is a tensorization for entropy. For an edge $e$, write $Ent_e X$ for the entropy of $X$ considered only as a function of $t_e$ (with all other weights fixed). The version below is \cite[Theorem~2.3]{DHS}.
\begin{prop}\label{prop: tensor}
If $X\in L^2$ is a non-negative function of $(t_e)$ then
\[
Ent~X \leq \sum_e \mathbb{E} Ent_e~X\ .
\]
\end{prop}

\section{Bound on influences}\label{sec: influences}
To bound the sum $\sum_{k=1}^\infty \left( \mathbb{E}|V_k|\right)^2$ we start with a simple lemma from \cite[Lemma~5.2]{DHS}. For a given edge-weight configuration $(t_e)$ and edge $e$, let $(t_{e^c},r)$ denote the configuration with value $t_f$ if $f \neq e$ and $r$ otherwise. Let $T_z(t_{e^c},r)$ be the variable $T_z = T(z,z+x)$ in the configuration $(t_{e^c},r)$ and define $Geo(z,z+x)$ as the set of edges in the intersection of all geodesics from $z$ to $z+x$.
\begin{lma}\label{lem: DHS_lem}
For $e \in \mathcal{E}^d$, the random variable
\[
D_{z,e} := \sup \left[ \{r \geq 0 : e \text{ is in a geodesic from }z \text{ to } z+x \text{ in } (t_{e^c},r)\} \cup \{0\} \right]
\]
has the following properties almost surely.
\begin{enumerate}
\item $D_{z,e} < \infty$.
\item For $0 \leq s \leq t$,
\[
T_z(t_{e^c},t) - T_z(t_{e^c},s) = \min \{t-s, (D_{z,e}-s)_+\}\ .
\]
\item For $0 \leq s < D_{z,e},~e \in Geo(z,z+x)$ in $(t_{e^c},s)$.
\end{enumerate}
\end{lma}

\noindent
It is important to note that $D_{z,e}$, and therefore $T_z(t_{e^c},t)-T_z(t_{e^c},s)$, does not depend on $t_e$.

We need one more lemma from \cite{DHS} bounding the length of geodesics. Let $\mathcal{G}$ be the set of all finite self-avoiding geodesics.
\begin{lma}\label{lem: geodesic_length}
Assuming $\mathbb{E}t_e^2<\infty$, there exists $\newconstant\label{C_5}$ such that for all finite $E \subset \mathcal{E}^d$,
\[
\mathbb{E} \max_{\gamma \in \mathcal{G}} \#(E \cap \gamma) \leq \oldconstant{C_5} \text{diam }E \ .
\]
\end{lma}

With these two tools we can bound the influences in the denominator of the logarithm of \eqref{eq: BR_lower_bound}. The following proof is very similar to the one of Benaim-Rossignol \cite[Theorem~4.2]{BR}.

\begin{prop}\label{prop: V_k_bound}
Assuming $\mathbb{E}t_e^2 < \infty$, there exists $\newconstant\label{C_99}$ such that
\[
\sum_{k=1}^\infty (\mathbb{E} |V_k|)^2 \leq \oldconstant{C_99} \lambda^2 \mathbb{E} ((1+e^{\lambda t_e})t_e)^2  \|x\|_1^{\frac{2+\zeta(1-d)}{2}}  \mathbb{E} e^{2\lambda F_m} \text{ for all }~x \in \mathbb{Z}^d\ .
\]
This inequality holds for any $\lambda$ for which the left side is defined.
\end{prop}

\noindent
Under \eqref{eq: exponential_moments}, we require $\lambda \in [0,\alpha]$ for the left side to be defined. Under \eqref{eq: two_moments}, one can take $\lambda \leq 0$.
\begin{proof}
Let $F_m^{(k)}$ be the variable $F_m$ with the edge weight $t_{e_k}$ replaced by an independent copy $t_{e_k}'$. Then we can give the upper bound
\begin{align}
\mathbb{E}|V_k| \leq \mathbb{E}\left| e^{\lambda F_m}-e^{\lambda F_m^{(k)}}\right| &= 2 \mathbb{E} \left( e^{\lambda F_m^{(k)}}-e^{\lambda F_m} \right)_+ \label{eq: first_rep}\\
&= 2 \mathbb{E} \left( e^{\lambda F_m}-e^{\lambda F_m^{(k)}} \right)_+ \label{eq: second_rep}\ .
\end{align}
We will use \eqref{eq: first_rep} when $\lambda >0$ and \eqref{eq: second_rep} when $\lambda \leq 0$. With these restrictions, the integrands in both cases above are only nonzero when $F_m^{(k)} > F_m$. 
Apply the mean value theorem to get
\[
\mathbb{E}|V_k| \leq 
\begin{cases}
\lambda~\mathbb{E}(e^{\lambda F_m^{(k)}} (F_m^{(k)} - F_m))_+ & \text{ when } \lambda>0 \\
|\lambda|~ \mathbb{E}(e^{\lambda F_m} (F_m^{(k)} - F_m))_+ & \text{ when } \lambda \leq 0
\end{cases}\ .
\]
To combine these, when $\lambda >0$, we use $F_m^{(k)} \leq F_m + t_{e_k}'$ to find $e^{\lambda F_m^{(k)}} \leq e^{\lambda F_m} e^{\lambda t_{e_k}'}$, so we obtain for both cases
\[
\mathbb{E}|V_k| \leq |\lambda|~ \mathbb{E} \left[ e^{\lambda F_m} \left[ (1+e^{\lambda t_{e_k}'})(F_m^{(k)}-F_m)_+ \right] \right]\ .
\]

By Cauchy-Schwarz, we have the following two bounds:
\begin{align}
\sum_{k=1}^\infty  \mathbb{E}|V_k| &\leq \sqrt{\lambda^2 \mathbb{E}e^{2\lambda F_m} \mathbb{E} \left( \sum_{k=1}^\infty (1+e^{\lambda t_{e_k}'}) \left(F_m^{(k)} - F_m \right)_+\right)^2} \text{ and } \label{eq: sum}\\
\mathbb{E}|V_k| &\leq \sqrt{\lambda^2 \mathbb{E}e^{2\lambda F_m} \mathbb{E}(1+e^{\lambda t_{e_k}'})^2\left(F_m^{(k)} - F_m\right)_+^2} \label{eq: term}\ .
\end{align}

We will bound these terms using Lemma~\ref{lem: DHS_lem}. Write 
\[
\mathbb{E} (1+ e^{\lambda t_{e_k}'})^2 (F_m^{(k)}-F_m)_+^2 = \mathbb{E} (1+e^{\lambda t_{e_k}'})^2 \left( \frac{1}{\# B_m} \sum_{z \in B_m} (T_z^{(k)}-T_z) \right)_+^2\ .
\]
Here $T_z^{(k)}$ is the variable $T_z$ in the configuration in which the $k$-th edge-weight $t_{e_k}$ is replaced by the independent copy $t_{e_k}'$. By convexity of the function $x \mapsto (x_+)^2$, we obtain the bound
\[
\frac{1}{\# B_m} \sum_{z \in B_m} \mathbb{E} (1+e^{\lambda t_{e_k}'})^2 (T_z^{(k)} - T_z)_+^2\ .
\]
By Lemma~\ref{lem: DHS_lem}, $(T_z^{(k)} - T_z)_+ \leq \mathbf{1}_{\{t_{e_k} < D_{z,e_k}\}} \times (t_{e_k}'- t_{e_k})_+$, so
\begin{align*}
\mathbb{E} (1+e^{\lambda t_{e_k}'})^2(F_m^{(k)}-F_m)_+^2 &\leq \frac{1}{\# B_m} \sum_{z \in B_m} \mathbb{E}((1+e^{\lambda t_{e_k}'})t_{e_k}')^2 \mathbf{1}_{\{t_{e_k} < D_{z,e_k}\}} \\
&\leq \mathbb{E}((1+e^{\lambda t_e})t_e)^2 \frac{1}{\# B_m} \sum_{z \in B_m} \mathbb{P}(e_k \in Geo(z,z+x))\ .
\end{align*}

By translation invariance, the final probability equals $\mathbb{P}(e_k - z \in Geo(0,x))$:
\begin{align}
\mathbb{E} (1+e^{\lambda t_{e_k}'})^2 (F_m^{(k)}-F_m)_+^2 &\leq \frac{\mathbb{E}((1+e^{\lambda t_e})t_e)^2}{\#B_m} \mathbb{E} \#\{e_k - z \in Geo(0,x) : z \in B_m\} \nonumber \\
&\leq \newconstant \|x\|_1^{\zeta(1-d)} \mathbb{E}((1+e^{\lambda t_e})t_e)^2 \label{eq: ingredient_1}\ .
\end{align}
We have used the assumption $\mathbb{E}t_e^2<\infty$ and Lemma~\ref{lem: geodesic_length} to bound the expectation above. After incorporating the factor $\lambda^2 \mathbb{E}e^{2\lambda F_m}$, this is our bound for \eqref{eq: term}.

For \eqref{eq: sum}, write
\begin{equation}\label{eq: ingredient_2}
\mathbb{E}\left( \sum_{k=1}^\infty (1+e^{\lambda t_{e_k}'}) (F_m^{(k)} - F_m)_+ \right)^2 \leq \frac{1}{\# B_m} \sum_{z \in B_m} \mathbb{E} \left( \sum_{k=1}^\infty ((1+e^{\lambda t_{e_k}'}) t_{e_k}')\mathbf{1}_{\{t_{e_k} < D_{z,e_k}\}} \right)^2\ .
\end{equation}
The expectation equals
\begin{align*}
&\sum_{k=1}^\infty \sum_{j=1}^\infty \mathbb{E} (1+e^{\lambda t_{e_k}'})(1+e^{\lambda t_{e_j}'}) t_{e_k}'t_{e_j}' \mathbf{1}_{\{t_{e_k}<D_{z,e_k},t_{e_j}<D_{z,e_j}\}} \\
\leq~& \mathbb{E}\left((1+e^{\lambda t_e})t_e\right)^2 \sum_{k=1}^\infty \sum_{j=1}^\infty \mathbb{P}(t_{e_k}<D_{z,e_k},t_{e_j}<D_{z,e_j}) \leq \mathbb{E}\left((1+e^{\lambda t_e})t_e\right)^2 \mathbb{E} \#Geo(z,z+x)^2\ .
\end{align*}
By \eqref{eq: two_length_bound} and $\mathbb{E}t_e^2 < \infty$, the last expression is bounded by $\newconstant \|x\|_1^2 \mathbb{E}((1+e^{\lambda t_e})t_e)^2$.

We can now finish the proof with this bound and \eqref{eq: ingredient_1}:
\[
\sum_{k=1}^\infty (\mathbb{E}|V_k|)^2 \leq \sup_j \mathbb{E}|V_j| \sum_{k=1}^\infty \mathbb{E}|V_k| \leq  \newconstant \lambda^2 \|x\|_1^{\frac{2+\zeta(1-d)}{2}}\mathbb{E}((1+e^{\lambda t_e})t_e)^2 \mathbb{E} e^{2 \lambda F_m}\ .
\]
\end{proof}

\section{Entropy bound}
The purpose of the present section is to give an intermediate upper bound for the sum of entropy terms in the left side of \eqref{eq: BR_lower_bound}. Namely we will prove the following inequality, recalling that $F$ is the distribution function of $t_e$.
\begin{thm}\label{thm: entropy_bound}
For some $\newconstant\label{D_3}>0$ independent of $\lambda$,
\[
\sum_{k=1}^\infty Ent(V_k^2) \leq \lambda^2 \frac{\oldconstant{D_3}C_\lambda}{\# B_m} \sum_{z \in B_m} \mathbb{E}\left[e^{2\lambda F_m} \sum_{e \in Geo(z,z+x)} (1-\log F(t_e)) \right] \text{ for all } x \in \mathbb{Z}^d\ .
\]
The constant $C_\lambda$ is determined as follows:
\begin{enumerate}
\item Assuming \eqref{eq: exponential_moments} and $\lambda \in [0,\alpha/2)$, $C_\lambda = Ent_e(t_e e^{\lambda t_e})^2 + \mathbb{E}[t_e e^{\lambda t_e}]^2$.
\item Assuming \eqref{eq: two_moments} and $\lambda \leq 0$, $C_\lambda = (\mathbb{E}e^{2\lambda t_e})^{-1}$.
\end{enumerate}
\end{thm}

We will prove this in a couple of steps. First we use the Bernoulli encoding from \cite{DHS} to give an upper bound (Lemma~\ref{lem: bern} below) in terms of discrete derivatives relative to Bernoulli sequences. Next we split into two cases, $\lambda \geq 0$ and $\lambda \leq 0$. The first is handled in Proposition~\ref{prop: intermediate_1} and the second in Proposition~\ref{prop: intermediate}. These three results will prove Theorem~\ref{thm: entropy_bound}.

\subsection{Bernoulli encoding}
\label{sec: benc}
We will now view our edge variables as the push-forward of Bernoulli sequences. Specifically, for each edge $e$, let $\Omega_e$ be a copy of $\{0,1\}^\mathbb{N}$ with the product sigma-algebra. Let also $\Omega_B = \prod_e \Omega_e$ with product sigma-algebra and measure $\pi := \prod_e \pi_e$, where $\pi_e$ is a product of the form $\prod_{j \geq 1} \pi_{e,j}$ with $\pi_{e,j}$ uniform on $\{0,1\}$. An element of $\Omega_B$ will be denoted
\[
\omega_B = \left\{ \omega_{e,j} : e \in \mathcal{E}^d,~j \geq 1 \right\}\ .
\]
For fixed $e$, and $\omega_e \in \Omega_e$, define the random variable
\begin{equation}
\label{eqn: Uedef}
U_e(\omega_e)=\sum_{j=1}^\infty \frac{\omega_{e,j}}{2^j}.
\end{equation}
Under $\pi$, for each $e$, $U_e(\omega_e)$ is uniformly distributed on $[0,1]$, and the collection $\{U_e(\omega_e)\}_{e\in \mathcal{E}^d}$ is independent. 

We denote by $F(x)$, $x\ge 0$ the distribution function of $\mu$, and by $I$ the infimum of the support:
\[I=\inf\{x:F(x)>0\}.\] 
The right-continuous inverse of $F$ is
\[F^{-1}(y) = \inf\{x: F(x) \ge y\}.\]
If $U$ is uniformly distributed on $[0,1]$, we have 
\begin{equation}
\label{eqn: uniform}
\mathbb{P}(F^{-1}(U)\le x) = \mathbb{P}(U\le F(x))= F(x),
\end{equation}
that is, $F^{-1}(U)$ has distribution $\mu$. Defining the measurable map $\varphi_e:\Omega_e \to \mathbb{R}$ by
\begin{equation}\label{eqn: Tedef}
\varphi_e(\omega_e) = F^{-1}(U_e(\omega_e)),
\end{equation}
we see that the distribution of $\varphi_e$ under $\pi$ is $\mu$. Moreover, since $\varphi_e$ is a composition of monotone functions of the the $\omega_{e,j}$, we have the important monotonicity property:
\begin{equation}
\label{eqn: monotone}
\varphi_e(\omega) \leq \varphi_e(\hat \omega) \quad  \text{ if } \quad \omega_j \leq \hat \omega_j \text{ for all } j\ge 1.
\end{equation}

Define the product map $\varphi:= \prod_e \varphi_e: \Omega_B \to \Omega=[0,\infty)^{\mathcal{E}^d}$ with action
\[
\varphi(\omega_B) = (\varphi_e(\omega_e) : e \in \mathcal{E}^d)\ .
\]
By \eqref{eqn: uniform}, $\pi \circ \varphi^{-1} = \mathbb{P}$.

In what follows, we will consider functions $f$ (in particular, $G=e^{\lambda F_m}$) on the original space $\Omega$ as functions on $\Omega_B$, through the map $\varphi$. We will suppress mention of this in the notation and, for instance, write $f(\omega_B)$ to mean $f \circ \varphi(\omega_B)$. We will estimate discrete derivatives, so for a function $f: \Omega_B \to \mathbb{R}$, set
\[
\left(\Delta_{e,j} f\right)(\omega_B) = f(\omega_B^{e,j,+}) - f(\omega_B^{e,j,-})\ ,
\]
where $\omega_B^{e,j,+}$ agrees with $\omega_B$ except possibly at $\omega_{e,j}$, where it is $1$, and $\omega_B^{e,j,-}$ agrees with $\omega_B$ except possibly at $\omega_{e,j}$, where it is 0. We now view $G = e^{\lambda F_m}$ as a function of sequences of Bernoulli variables, as in \cite{DHS}. Then, exactly the same proof as in \cite[Lemma~6.3]{DHS} gives
\begin{lma}\label{lem: bern}
Assume \eqref{eq: exponential_moments} and $\lambda \in [0,\alpha/2)$ or \eqref{eq: two_moments} and $\lambda \leq 0$. We have the following inequality:
\[
\sum_{k=1}^\infty Ent(V_k^2) \leq \sum_{e,j} \mathbb{E}_\pi\left(\Delta_{e,j}G\right)^2\ .
\]
\end{lma}

\subsection{Derivative bound: positive exponential}
For the next derivative bounds we continue with $G = e^{\lambda F_m}$ and set $H$ as the derivative of $G$; that is, $H = \lambda e^{\lambda F_m}$.

\begin{thm}\label{prop: intermediate_1}
Assume \eqref{eq: exponential_moments}. For some $\newconstant\label{Cmu}>0$, $C_\lambda = Ent_e(t_e e^{\lambda t_e})^2 + \mathbb{E}[t_e e^{\lambda t_e}]^2$ and all $\lambda \in [0,\alpha/2)$,
\[
\sum_{e,j} \mathbb{E}_\pi \left( \Delta_{e,j} G \right)^2 \leq \lambda^2 \frac{\oldconstant{Cmu} C_\lambda}{\#B_m} \sum_{z\in B_m} ~\mathbb{E} \left[ e^{2\lambda F_m} \sum_{e \in Geo(z,z+x)} (1-\log F(t_e)) \right] \text{ for all } x \in \mathbb{Z}^d\ .
\]
\end{thm}

\begin{proof}
We first note that
\[\mathbb{E}_{\pi_e}(\Delta_{e,j} G)^2 = 2\mathbb{E}_{\pi_e}(G(\omega_B)-G(\omega_B^{e,j,-}))_+^2.\]

By the mean value theorem and monotonicity of $H$ in the variable $\omega_e$, we have 
\[0\le G(\omega_B)-G(\omega_B^{e,j,-}) \le H(\omega_B)(F_m(\omega_B)-F_m(\omega_B^{e,j,-}))^2_+.\]
Convexity of $x\mapsto x^2$ gives the bound:
\begin{equation}
\label{eqn: DeltaGbound}
\mathbb{E}_{\pi_e}(\Delta_{e,j}G)^2 \le \frac{2}{\# B_m} \sum_{z\in B_m} \mathbb{E}_{\pi_e}H^2(\omega_{e^c},U_e(\omega_e))\left(T_z(\omega_{e^c},U_e(\omega_e))-T_z(\omega_{e^c},U_e(\omega_e)-2^{-j})\right)^2\mathbf{1}_{\{U_e(\omega_e)\ge 2^{-j}\}}.
\end{equation}
Here we have written $H^2(\omega_{e^c},U_e(\omega_e))$ and $T_z(\omega_{e^c},U_e(\omega_e))$ to denote $H^2$ and $T_z$, taken as functions of the Bernoulli variables $\omega_{e'}$, $e\neq e'$, and the (independent) uniform variable $U_e(\omega_e)$. 

To estimate the expectation over $U_e(\omega_e)$, we use the following key lemma. The abstract formulation here was suggested by a referee for an earlier version of the current paper. The referee's proof below summarizes the technical ideas present in a computation using Bernoulli variables that appeared in the previous version, and which more closely followed our argument in \cite{DHS}.
\begin{lma} \label{lma: uniformvar}
Let $a,\tau\in [0,1]$. Suppose $h$ and $f$ are nonnegative, non-decreasing on $[0,1]$, and that $f$ is constant on $[a,1]$. If $\tau \le 1/2$, then 
\begin{equation}\label{eqn: a>1/2}
\int_\tau^1 h(x)(f(x)-f(x-\tau))^2\,\mathrm{d}x \le \int_0^1 h(x)f^2(x) \mathbf{1}_{\{x\ge 1-\tau\}}\,\mathrm{d}x.
\end{equation}
Moreover:
\begin{enumerate}
\item If $a\le \tau \le \frac{1}{2}$,
\begin{equation} \label{eqn: a<tau}
\int_\tau^1h(x)(f(x)-f(x-\tau))^2\,\mathrm{d}x\le 2a\int_0^1 h(x)f^2(x)\,\mathrm{d}x.
\end{equation}
\item If $\tau\le a\le \frac{1}{2}$,
\begin{equation}
\label{eqn: a>tau}
\int_\tau^1 h(x)(f(x)-f(x-\tau))^2\,\mathrm{d}x \le 2\tau \int_0^1 h(x)f^2(x)\,\mathrm{d}x.
\end{equation}
\end{enumerate}
\begin{proof}
If $\tau \le 1/2$, then
\begin{align*}
\int_\tau^1 h(x)(f(x)-f(x-\tau))^2\,\mathrm{d}x &\le \int_{\tau}^1 h(x)(f^2(x)-f^2(x-\tau))\,\mathrm{d}x\\
&\le \int_\tau^1 h(x)f^2(x)-h(x-\tau)f^2(x-\tau)\,\mathrm{d}x,
\end{align*}
where we have used the monotonicity of $f$ in the first inequality, and $h$ in the second.
Next,
\begin{align*}
 \int_\tau^1 h(x)f^2(x)-h(x-\tau)f^2(x-\tau)\,\mathrm{d}x &= \int_\tau^1 h(x)f^2(x)\, \mathrm{d}x - \int_0^{1-\tau}h(x)f^2(x)\,\mathrm{d}x\\
&\le \int_{1-\tau}^1 h(x)f^2(x)\,\mathrm{d}x\\
&= \int_0^1 h(x)f^2(x) \mathbf{1}_{\{x\ge 1-\tau\}}\,\mathrm{d}x.
\end{align*}
We have used the non-negativity of $h(x)f^2(x)$ to drop the integral over $[0,\tau]$. This shows \eqref{eqn: a>1/2}.

If $a\le \tau\le \frac{1}{2}$, since $f$ is constant over $[a,1]$,
\[\int_\tau^1 h(x)(f(x)-f(x-\tau))^2\,\mathrm{d}x = \int_\tau^{a+\tau}h(x)(f(x)-f(x-\tau))^2\,\mathrm{d}x.\]
By monotonicity, the right side is no bigger than
\[\int_{\tau}^{a+\tau} h(x)f^2(x)\,\mathrm{d}x = \int_{\tau}^1 h(x)f^2(x)\mathbf{1}_{\{x\le a+\tau\}}\,\mathrm{d}x.\]
The Chebyshev association inequality \cite[Theorem~2.14]{BLM} now gives \eqref{eqn: a<tau}:
\begin{align*}
\int_{\tau}^1 h(x)f^2(x)\mathbf{1}_{\{x\le a+\tau\}}\,\mathrm{d}x &\le \frac{1}{1-\tau}\int^1_\tau h(x)f^2(x)\,\mathrm{d}x \int_\tau^1 \mathbf{1}_{\{x\le a+\tau\}}\,\mathrm{d}x\\
&= \frac{a}{1-\tau}\int_\tau^1 h(x)f^2(x)\,\mathrm{d}x\\
&\le 2a\int_0^1 h(x)f^2(x)\,\mathrm{d}x.
\end{align*}

When $\tau \le a\le \frac{1}{2}$, we again have
\[\int_\tau^1 h(x)(f(x)-f(x-\tau))^2\,\mathrm{d}x=\int_\tau^{a+\tau} h(x)(f(x)-f(x-\tau))^2\,\mathrm{d}x.\]
Expanding the square and using monotonicity as in the case $\tau \ge 1/2$, we find
\begin{align*}
\int_\tau^{a+\tau} h(x)(f(x)-f(x-\tau))^2\,\mathrm{d}x &\le \int_\tau^{a+\tau}h(x)f^2(x)\,\mathrm{d}x-\int_0^ah(x)f^2(x)\,\mathrm{d}x\\
&= \int_a^{a+\tau} h(x)f^2(x)\,\mathrm{d}x - \int_0^\tau h(x)f^2(x)\,\mathrm{d}x\\
&\le \int_a^1 h(x)f^2(x)\mathbf{1}_{\{x\le a+\tau\}}\,\mathrm{d}x\\
&\le \frac{\tau}{1-a}\int_0^1 h(x)f^2(x)\,\mathrm{d}x \\
&\le 2\tau \int_0^1h(x)f^2(x)\,\mathrm{d}x.
\end{align*}
In the second-to-last step, we have used Chebyshev's association inequality. 
\end{proof}
\end{lma}

We use Lemma \ref{lma: uniformvar} to obtain an estimate for the sum $\sum_{e,j} \mathbb{E}_\pi\left(\Delta_{e,j}G\right)^2$. The result we are after is
\begin{lma}
For all $\lambda \in [0,\alpha/2)$ and $e$,
\begin{equation}\label{eqn: lemmadisplay}
\mathbb{E}_{\pi_e} \sum_{j=1}^\infty (\Delta_{e,j} G)^2 \leq 8 \frac{1}{\# B_m} \sum_{z \in B_m} F(D_{z,e}^-)(1-\log_2 F(D_{z,e}^-)) \mathbb{E}_{\pi_e}[L_z (1+N)]\
\end{equation}
where $L_z= H^2(\omega_{e^c},U_e(\omega_e))\min\{D_{z,e},\varphi_e(\omega_e)\}^2$, and
\[N= \log_2 \frac{1}{1-U_e(\omega_e)}.\]
\end{lma}
\begin{proof}
Write:
\begin{align}
&\quad \mathbb{E}_{\pi_e} H^2(\omega_{e^c},U_e(\omega_e))\left(T_z(\omega_{e^c},U_e(\omega_e))-T_z(\omega_{e^c},U_e(\omega_e)-2^{-j})\right)^2\mathbf{1}_{\{U_e(\omega_e)\ge 2^{-j}\}} \nonumber \\
=&\quad \int_0^1  H^2(\omega_{e^c},x)\left(T_z(\omega_{e^c},x)-T_z(\omega_{e^c},x-2^{-j})\right)^2\mathbf{1}_{\{x\ge 2^{-j}\}}\,\mathrm{d}x.
\end{align}
We will apply Lemma \ref{lma: uniformvar} to the quantity on the right in \eqref{eqn: DeltaGbound} with $h(x)=H^2(\omega_{e^c},x)$, $f(x)=T_z(\omega_{e^c},x)-T_z(\omega_{e^c},0)$, $a=F(D_{z,e}^-)$, and $\tau = 2^{-j}$.

For $j$ such that $F(D_{z,e}^-) < 2^{-j}$ ($a < \tau$), we use \eqref{eqn: a<tau}:
\begin{align*}
&\, \int_0^1  H^2(\omega_{e^c},x)\left(T_z(\omega_{e^c},x)-T_z(\omega_{e^c},x-2^{-j})\right)^2\mathbf{1}_{\{x\ge 2^{-j}\}}\,\mathrm{d}x \\
\le & \, 2F(D_{z,e}^-)\int_0^1  H^2(\omega_{e^c},x)(T_z(\omega_{e^c},x)-T_z(\omega_{e^c},0))^2\,\mathrm{d}x.
\end{align*}
By Lemma \ref{lem: DHS_lem},
\[0\le T_z(\omega_{e^c},x)-T_z(\omega_{e^c},0) \le \min\{D_{z,e},\varphi_e(\omega_e)\},\]
and so
\[\int_0^1  H^2(\omega_{e^c},x)(T_z(\omega_{e^c},x)-T_z(\omega_{e^c},0))^2\,\mathrm{d}x \le  \int_0^1 L_z(\omega_{e^c},x)\,\mathrm{d}x.\]
If $j$ is such that $2^{-j} \le F(D_{z,e}^-) \le \frac{1}{2}$ ($\tau\le a \le \frac{1}{2}$), we apply \eqref{eqn: a>tau} and obtain
\begin{align*}
&\, \int_0^1  H^2(\omega_{e^c},x)\left(T_z(\omega_{e^c},x)-T_z(\omega_{e^c},x-2^{-j})\right)^2\mathbf{1}_{\{x\ge 2^{-j}\}}\,\mathrm{d}x \\
\le & \, 2\cdot 2^{-j} \int_0^1  L_z(\omega_{e^c},x)\,\mathrm{d}x
\end{align*}

We now sum over $j$ in \eqref{eqn: DeltaGbound}: in case $F(D^-_{z,e}) \le \frac{1}{2}$, we obtain
\begin{align}
&\,\sum_{j=1}^\infty \mathbb{E}_{\pi_e}(\Delta_{e,j}G)^2 \nonumber \\
\le &\, \sum_{j=1}^\infty \frac{2}{\# B_m} \sum_{z\in B_m} \mathbb{E}_{\pi_e}H^2(\omega_{e^c},U_e(\omega_e))\left(T_z(\omega_{e^c},U_e(\omega_e))-T_z(\omega_{e^c},U_e(\omega_e)-2^{-j})\right)^2\mathbf{1}_{\{U_e(\omega_e)\ge 2^{-j}\}} \nonumber \\
\le &\,\frac{4}{\# B_m} \sum_{z\in B_m} F(D_{z,e}^-)\sum_{j: 2^{-j} > F(D_{z,e}^-)} \int_0^1 L_z(\omega_{e^c},x)\,\mathrm{d}x \nonumber \\
&\,  + \frac{4}{\# B_m} \sum_{z\in B_m} \sum_{j: 2^{-j} \le F(D_{z,e}^-)} 2^{-j} \int_0^1  L_z(\omega_{e^c},x)\,\mathrm{d}x \nonumber \\
\le &\,\frac{8}{\# B_m} \sum_{z\in B_m} F(D_{z,e}^-)(1-\log_2 F(D_{z,e}^-))\int_0^1 L_z(\omega_{e^c},x)\,\mathrm{d}x. \label{eqn: jless1/2}
\end{align}

If instead $F(D_{z,e}^-)\ge \frac{1}{2}$, we use \eqref{eqn: a>1/2} in the sum over $j$ such that $2^{-j} \le F(D_{z,e}^-)$:
\begin{align}
&\, \sum_{j: 2^{-j} \le F(D_{z,e}^-)} \mathbb{E}_{\pi_e}H^2(\omega_{e^c},U_e(\omega_e))\left(T_z(\omega_{e^c},U_e(\omega_e))-T_z(\omega_{e^c},U_e(\omega_e)-2^{-j})\right)^2\mathbf{1}_{\{U_e(\omega_e)\ge 2^{-j}\}} \nonumber \\
= &\, \sum_{j: 2^{-j} \le F(D_{z,e}^-)} \int_0^1 H^2(\omega_{e^c},x)\left(T_z(\omega_{e^c},x)-T_z(\omega_{e^c},x-2^{-j})\right)^2\mathbf{1}_{\{x\ge 2^{-j}\}}\,\mathrm{d}x \nonumber \\
\le& \, \int_0^1 L_z(\omega_{e^c},x) \sum_{j:2^{-j}\le F(D_{z,e}^-)} \mathbf{1}_{\{x\ge 1-2^{-j}\}}\,\mathrm{d}x \nonumber \\
\le& \, \int_0^1 L_z(\omega_{e^c},x) \log_2 \frac{1}{1-x}\,\mathrm{d}x \nonumber \\
\le & \, 2F(D^-_{z,e})  \int_0^1 L_z(\omega_{e^c},x) \log_2 \frac{1}{1-x}\,\mathrm{d}x. \label{eqn: -log(1-x)}
\end{align}
On the other hand, for the sum over $j$ such that $2^{-j} > F(D_{z,e}^-)$, we drop the indicator $\mathbf{1}_{\{x\ge 1-2^{-j}\}}$:
\begin{align}
&\, \sum_{j: 2^{-j} > F(D_{z,e}^-)} \mathbb{E}_{\pi_e}H^2(\omega_{e^c},U_e(\omega_e))\left(T_z(\omega_{e^c},U_e(\omega_e))-T_z(\omega_{e^c},U_e(\omega_e)-2^{-j})\right)^2\mathbf{1}_{\{U_e(\omega_e)\ge 2^{-j}\}} \nonumber \\
= &\, \sum_{j: 2^{-j} >F(D_{z,e}^-) } \int_0^1 H^2(\omega_{e^c},x)\left(T_z(\omega_{e^c},x)-T_z(\omega_{e^c},x-2^{-j})\right)^2\mathbf{1}_{\{x\ge 2^{-j}\}}\,\mathrm{d}x \nonumber \\
\le &\, -\log_2 F(D_{z,e}^-)\int_0^1 L_z(\omega_{e^c},x) \,\mathrm{d}x\nonumber \\
\le &\, -2F(D^-_{z,e})\log_2 F(D^-_{z,e})\int_0^1 L_z(\omega_{e^c},x) \,\mathrm{d}x. \label{eqn: logF}
\end{align}
Combining \eqref{eqn: jless1/2}, \eqref{eqn: -log(1-x)} and \eqref{eqn: logF}, the lemma follows.
\end{proof}

We now finish the proof of Theorem \ref{prop: intermediate_1}. Integrating \eqref{eqn: lemmadisplay} over variables $e'$, $e'\neq e$ and summing over $e$, and using the identity
\begin{align}
-F(y^-)\log F(y^-) &= -\int\mathbf{1}_{[I,y)}(x)\log F(y^-)\,\mu(\mathrm{d}x) \nonumber \\
&\le -\int \mathbf{1}_{[I,y)}(x)\log F(x)\,\mu(\mathrm{d}x), \label{eqn: intident}
\end{align} 
we obtain:
\begin{equation}\label{eq: broom_handle}
\sum_e \mathbb{E}_\pi \sum_{j=1}^\infty (\Delta_{e,j} G)^2 \leq \frac{8}{\# B_m} \sum_{z \in B_m} \sum_e \mathbb{E}_{e^c} \left[ \int_{[I,D_{z,e})} (1-\log_2 F(t_e))\mathbb{E}_{\pi_e} \left[ L_z (1+N)\right]~\mu(\text{d}t_e)  \right] \\
\end{equation}
Letting $F_{m,e^c}$ be $F_m$ evaluated at the configuration $(t_{e^c},0)$, we can bound $H^2 \leq \lambda^2 e^{2\lambda F_{m,e^c}} e^{2\lambda t_e}$, so
\begin{align*}
\mathbb{E}_{\pi_e} \left[ L_z(N+1) \right] \leq \mathbb{E}_{\pi_e}\left[ H^2 t_e^2 (N+1) \right] &\leq \lambda^2 e^{2 \lambda F_{m,e^c}} \mathbb{E}_{\pi_e} \left[ (t_e e^{\lambda t_e})^2 (N+1) \right] \\
&\leq H^2 \mathbb{E}_{\pi_e} \left[ (t_e e^{\lambda t_e})^2 (N+1) \right]\ .
\end{align*}
Now since $\lambda < \alpha/2$, $Ent_e(t_ee^{\lambda t_e})^2 < \infty$, so we use Proposition~\ref{prop: variation} to bound the expectation by
\[
2Ent_e (t_e e^{\lambda t_e})^2 + 2 \mathbb{E}_\mu (t_e e^{\lambda t_e})^2 \log \mathbb{E}_{\pi_e} e^{(N+1)/2}\ .
\]
$N$ has exponential distribution with mean $1/\log 2$, so this is bounded by $\newconstant C_\lambda$ independently of $e$. 

Returning to \eqref{eq: broom_handle}, note that by Lemma~\ref{lem: DHS_lem}, if $t_e < D_{z,e}$ then $e$ is in $Geo(z,z+x)$. So applying the bound on $\mathbb{E}_{\pi_e}[L_z N]$, we obtain for some $\newconstant\label{C_70}$
\begin{align*}
\sum_{e,j} \mathbb{E}_\pi \left( \Delta_{e,j} G \right)^2 &\leq \frac{\oldconstant{C_70} C_\lambda}{\# B_m} \sum_{z \in B_m} \sum_e \mathbb{E} \left[ (1-\log F(t_e)) H^2 \mathbf{1}_{\{I \leq t_e < D_{z,e}\}} \right] \\
&\leq \lambda^2 \frac{\oldconstant{C_70} C_\lambda }{\# B_m} \sum_{z \in B_m} \mathbb{E} \left[ e^{2\lambda F_m} \sum_{e \in Geo(z,z+x)} (1-\log F(t_e)) \right]\ .
\end{align*}
\end{proof}

\subsection{Derivative bound: negative exponential}

\begin{thm}\label{prop: intermediate}
Assume \eqref{eq: two_moments}. For some $\newconstant\label{C_42}>0$, $C_\lambda' = \mathbb{E}_{\mu}e^{2\lambda t_e}$ and all $\lambda \leq 0$,
\[
\sum_{e,j} \mathbb{E}_\pi \left( \Delta_{e,j} G \right)^2 \leq \lambda^2 \frac{\oldconstant{C_42}}{C_\lambda' \#B_m} \sum_{z\in B_m} ~\mathbb{E} \left[ e^{2\lambda F_m} \sum_{e \in Geo(z,z+x)} (1-\log F(t_e)) \right] \text{ for all } x \in \mathbb{Z}^d\ .
\]
\end{thm}

\begin{proof}

Write $\mathbb{E}_{e^c}$ for expectation relative to $\prod_{f \neq e} \pi_f$ and for any $i\geq 1$, let $\pi_{e,\geq i}$ be the measure $\prod_{k \geq i} \pi_{e,k}$. Further, for $j \geq 1$ write 
\[
\omega_B = (\omega_{e^c}, \omega_{e,< j}, \omega_{e,j}, \omega_{e,> j})\ ,
\]
where $\omega_{e^c}$ is the configuration $\omega_B$ projected on the coordinates $(\omega_{f,k} : f \neq e,~ k \geq 1)$, $\omega_{e,<j}$ is $\omega_B$ projected on the coordinates $(\omega_{e,k} : k < j)$ and $\omega_{e, > j}$ is $\omega_B$ projected on the coordinates $(\omega_{e,k} : k > j)$.

Then
\begin{align}
\mathbb{E}_\pi \left( \Delta_{e,j} G \right)^2=&~\mathbb{E}_{e^c} \mathbb{E}_{\pi_{e,1}} \cdots \mathbb{E}_{\pi_{e,j-1}} \left[ \mathbb{E}_{\pi_{e,\geq j}} \left( \Delta_{e,j} G \right)^2 \right] \nonumber \\
=& ~\mathbb{E}_{e^c} \left[ \frac{1}{2^{j-1}} \sum_{\sigma \in \{0,1\}^{j-1}} \left[ \mathbb{E}_{\pi_{e,\geq j}} \left( \Delta_{e,j} G (\omega_{e^c}, \sigma, \omega_{e,j}, \omega_{e, > j}) \right)^2 \right] \right] \label{eq: seconD_{z,e}q}\ ,
\end{align}
and the innermost term is
\begin{equation}\label{eq: sum_on_home}
\mathbb{E}_{\pi_{e, \geq j}} \left( G(\omega_{e^c},\sigma, 1, \omega_{e,>j}) - G(\omega_{e^c}, \sigma,0, \omega_{e,>j}) \right)^2 \ .
\end{equation}
Applying the mean value theorem, we get an upper bound of
\[
H^2(\omega_{e^c}, \vec 0) ~\mathbb{E}_{\pi_{e,\geq j}}  (F_m(\omega_{e^c},\sigma,1,\omega_{e,>j})-F_m(\omega_{e^c},\sigma,0,\omega_{e,>j}))^2\ ,
\]
where $\vec 0$ is the infinite sequence $(0, 0, \ldots)$. Convexity of $x \mapsto x^2$ gives the bound
\[
\frac{1}{\#B_m} \sum_{z \in B_m} H^2(\omega_{e^c},\vec 0)~\mathbb{E}_{\pi_{e,\geq j}}(\Delta_{e,j} T_z(\omega_{e^c},\sigma, \omega_{e,j},\omega_{e,>j}) )^2\ .
\]
Therefore
\begin{align*}
\sum_{j=1}^\infty \mathbb{E}_\pi(\Delta_{e,j}G)^2 &\leq \frac{1}{\# B_m} \sum_{z \in B_m} \mathbb{E}_{e^c} \left[ H^2(\omega_{e^c}, \vec 0) \sum_{j=1}^\infty \frac{1}{2^{j-1}} \left[ \sum_{\sigma \in \{0,1\}^{j-1}} \mathbb{E}_{\pi_{e,\geq j}} (\Delta_{e,j} T_z(\omega_{e^c},\sigma,\omega_{e,j},\omega_{e,>j}))^2 \right] \right] \\
& = \frac{1}{\# B_m} \sum_{z \in B_m} \mathbb{E}_{e^c} \left[ H^2(\omega_{e^c},\vec 0) \sum_{j=1}^\infty \mathbb{E}_{\pi_e}(\Delta_{e,j} T_z )^2 \right]\ .
\end{align*}

We have now isolated the term from \cite[(6.23)]{DHS}; there it is proved under \eqref{eq: two_moments} that
\[
\sum_{j=1}^\infty \mathbb{E}_{\pi_e} (\Delta_{e,j}T_z)^2 \leq \newconstant\label{C_23} F(D_{z,e}^-)(1-\log F(D_{z,e}^-)) \mathbf{1}_{\{I < D_{z,e}\}}\ .
\]
Thus we obtain
\begin{equation}\label{eq: almost_end_11}
\sum_{j=1}^\infty \mathbb{E}_\pi (\Delta_{e,j}G)^2 \leq \frac{\oldconstant{C_23}}{\#B_m} \sum_{z \in B_m} \mathbb{E}_{e^c} \left[ H^2 (\omega_{e^c},\vec 0)~F(D_{z,e}^-)(1-\log F(D_{z,e}^-))\mathbf{1}_{\{I<D_{z,e}\}} \right]\ .
\end{equation}
Use the bound
\[
H^2 \geq \lambda^2 e^{2\lambda t_e} e^{2\lambda F_m(t_{e^c},I)}\ ,
\]
which implies $H^2(\omega_{e^c},\vec 0) \leq \frac{\mathbb{E}_{\pi_e}H^2}{\mathbb{E}_\mu e^{2\lambda t_e}}$. Using the identity \eqref{eqn: intident}, this gives an upper bound for the right side of \eqref{eq: almost_end_11} when $\lambda \leq 0$:
\[
\frac{\oldconstant{C_23}}{\mathbb{E}_\mu e^{2\lambda t_e} \# B_m} \sum_{z \in B_m} \mathbb{E}_{e^c}  \left[ \left[ \mathbb{E}_{\pi_e} H^2 \right] \left[ \int (1-\log F(t_e)) \mathbf{1}_{[I,D_{z,e})}(t_e) ~\text{d}\mu(t_e)\right] \right]\ .
\]
Since $\lambda \leq 0$, $H^2 = \lambda^2 e^{2\lambda F_m}$ is decreasing in the variable $t_e$. However $(1-\log F(t_e)) \mathbf{1}_{[I,D_{z,e})}(t_e)$ is also decreasing in $t_e$. Therefore the Chebyshev association inequality gives an upper bound of
\[
\frac{\oldconstant{C_23}}{\mathbb{E}_\mu e^{2\lambda t_e} \# B_m} \sum_{z \in B_m} \mathbb{E} H^2 (1-\log F(t_e)) \mathbf{1}_{\{e \in Geo(z,z+x)\}}\ .
\]
Summing over edges $e$,
\[
\sum_{e,j} \mathbb{E}_\pi(\Delta_{e,j}G)^2 \leq \lambda^2 \frac{\oldconstant{C_23}}{\mathbb{E}e^{2\lambda t_e} \# B_m} \sum_{z \in B_m} \mathbb{E}\left[ e^{2\lambda F_m} \sum_{e \in Geo(z,z+x)} (1-\log F(t_e)) \right]\ .
\]

\end{proof}

\section{Control by lattice animals}\label{sec: animals}

The next step is to use the theory of greedy lattice animals to decouple and control the terms in the expectation of Theorem~\ref{thm: entropy_bound}. Specifically we will show
\begin{thm}\label{thm: next_step}
Assume \eqref{eq: exponential_moments} with $\lambda \in [0,\alpha/2)$ or \eqref{eq: two_moments} with $\lambda \leq 0$. For some $\newconstant\label{D_6}>0$,
\[
\sum_{k=1}^\infty Ent(V_k^2) \leq \lambda^2 \oldconstant{D_6}C_{\lambda} \left[ Ent(e^{2\lambda F_m}) + (1+ \mathbb{E}F_m) \mathbb{E}e^{2\lambda F_m} \right] \text{ for all } x \in \mathbb{Z}^d\ ,
\]
where $C_\lambda$ is from Theorem~\ref{thm: entropy_bound}.
\end{thm}

The theorem follows from inequalities \eqref{eq: positive} and \eqref{eq: negative}, which we now set out to prove. We begin by generating a new set of ``lattice animal weights'' from a given realization $(t_e)$; set
\[
w_e:= 1 - \log( F(t_e)) \quad \text{ for all } e \in \mathcal{E}^d\ . 
\]

\begin{prop}
\label{prop:expdist}
The collection $(w_e)$ is i.i.d. with
\[ 
\mathbb{E}\left(e^{w_e/2} \right)  < \infty\ .
\]
\end{prop}
\begin{proof}
If $u \in (0,1)$ we define $F^{-1}(u) = \inf\{ x : F(x) \geq u\}$, so that $F^{-1}(u) \leq x$ if and only if $u \leq F(x)$. In particular, $u \leq F(F^{-1}(u))$ for all $u$. If $U$ is uniformly distributed on $(0,1)$ then $F^{-1}(U)$ is distributed like $t_e$, so if $r \geq 1$,
\[
\mathbb{P}(w_e \geq r) = \mathbb{P}(F(F^{-1}(U)) \leq e^{1-r}) \leq \mathbb{P}(U \leq e^{1-r}) = e^{1-r}\ .
\]
This implies $\mathbb{E}e^{\lambda w_e}<\infty$ for all $\lambda < 1$.
\end{proof}

For a realization of $(t_e),$ consider the edge greedy lattice animal problem. For a connected subset of edges $\gamma \subseteq \mathcal{E}^d,$ define $N(\gamma) = \sum_{e \in \gamma} w_e$, and define the random variable
\[
N_n := \max_{\substack{\gamma: \# \gamma = n\\0 \in \gamma}} N(\gamma)
\]
(here the notation $0 \in \gamma$ means that $0$ is an endpoint of some edge in $\gamma$).

\begin{prop}
\label{prop:latticeanbound}
For each $\kappa > 0$, there exists $\beta > 0$ such that
\begin{equation*}
\sup_{n > 0} \frac{\log \mathbb{E} e^{\beta N_n}}{n} \leq \kappa\ .
\end{equation*}

\end{prop}

\begin{proof}
Recall
\[
\{N_n > \beta n\} = \bigcup_{\gamma: \, \# \gamma = n} \{N(\gamma) > \beta n\}\ ,
\]
where the union is over all lattice animals of  size $n$ containing the origin. Now, there exists a constant $\newconstant\label{D_7}$ such that the number of such lattice animals is bounded by $e^{\oldconstant{D_7} n}$. Therefore, letting $(w_i)$ be a sequence of i.i.d. random variables distributed as $w_e$,
\begin{align*}
\mathbb{P}\left(N_n > \beta n\right) \leq e^{\oldconstant{D_7} n}\, \mathbb{P}\left(\sum_{i=1}^n w_i > \beta n\right) &\leq  e^{\oldconstant{D_7} n - \beta n /2} \mathbb{E}e^{\sum_{i=1}^n w_i/2}\\
&=  e^{\oldconstant{D_7} n - \beta n /2} \left[\mathbb{E}e^{w_e/2}\right]^n\ .
\end{align*}
In particular, for all $\beta$ greater than some $\beta_0,$
\[
\mathbb{P}\left(N_n > \beta n\right) \leq e^{-\beta n/4}\ . 
\]
Now, for all $\lambda \in [0,1/8)$ and for each $n\geq 1$,
\begin{align*}
\mathbb{E}e^{\lambda N_n} = \lambda \int_0^\infty e^{\lambda x} \mathbb{P}\left(N_n \geq x\right) \, \mathrm{d}x &\leq e^{\beta_0 n} + \lambda \int_{\beta_0 n/\lambda}^{\infty} e^{\lambda x}  \mathbb{P}\left(N_n \geq x\right) \, \mathrm{d}x\\
&\leq e^{\beta_0 n} + n \int_{\beta_0}^\infty e^{-\Xi n}\mathrm{d} \Xi\\
&\leq e^{\beta_1 n}
\end{align*}
for some $\beta_1 < \infty.$
Now, since $\mathbb{E}e^{\lambda N_n / 2} \leq \left(\mathbb{E}e^{\lambda N_n}\right)^{1/2},$ the proof is complete.
\end{proof}

We now consider the first-passage model on $\mathbb{Z}^d.$
For any $x \in \mathbb{Z}^d$, $x \neq 0,$ let
\[
Y_x:= \sum_{e \in Geo(0,x)} w_e\ ;
\]
when we need to allow the starting point to vary as well, write
\[
Y_{z,x}:= \sum_{e \in Geo(z,z+x)} w_e\ .
\]

\subsection{The case $\lambda \geq 0$}
In this section, we consider the case of upper exponential concentration. For the remainder of this section, assume \eqref{eq: exponential_moments} and let $\lambda \in [0,\alpha/2)$.

Rephrase the bound from Theorem~\ref{thm: entropy_bound}:
\begin{equation}
\label{eq:uppertail_want_to_bd}
\sum_{k=1}^\infty Ent(V_k^2) \leq \lambda^2 \frac{\oldconstant{D_3} C_\lambda}{\#B_m} \sum_{z\in B_m} ~\mathbb{E} \left[ e^{2\lambda F_m} Y_{z,x} \right]\ .
\end{equation}
Applying Proposition~\ref{prop: variation} to the expectation on the right-hand side of \eqref{eq:uppertail_want_to_bd} and the fact that $Y_x = Y_{z,x}$ in distribution yields a bound for some $\newconstant\label{D_8}>0$ such that the expectation below exists:
\begin{equation}\label{eq:uppertail_after_holder}
\mathbb{E}\left[ e^{2\lambda F_m} Y_{z,x} \right] \leq \oldconstant{D_8}^{-1} Ent(e^{2 \lambda F_m}) + \oldconstant{D_8}^{-1}\mathbb{E}\left[ e^{2 \lambda F_m} \right]\, \log \mathbb{E} e^{\oldconstant{D_8} Y_x}\ .
\end{equation}
We focus our efforts on a bound for the second term of \eqref{eq:uppertail_after_holder}.

\begin{prop}
Assuming \eqref{eq: exponential_moments}, there exists $\oldconstant{D_8}>0$ such that
\[ 
\sup_{x \neq 0} \frac{\log \mathbb{E}e^{\oldconstant{D_8} Y_x}}{\|x\|_1} <\infty\ .
\]
\end{prop}
\begin{proof}
Recall that $G(0,x)$ is the maximal number of edges in a geodesic from $0$ to $x$. Then for constants $\iota_1,\iota_2>0$,
\begin{equation}\label{eq: no_details}
\mathbb{P}(Y_x \geq \iota_1 \|x\|_1) \leq \mathbb{P}(G(0,x) \geq \iota_2 \|x\|_1) + \mathbb{P}(N_{\iota_2\|x\|_1} \geq \iota_1\|x\|_1)\ .
\end{equation}
Therefore
\[
\mathbb{E}e^{\oldconstant{D_8}Y_x} = 1+\int_0^\infty \oldconstant{D_8} e^{\oldconstant{D_8} y} \mathbb{P}(Y_x \geq y)~\text{d}y
\]
can be bounded using \eqref{eq: no_details}, Proposition~\ref{prop:latticeanbound} and \eqref{eq: exponential_length_bound}.
\end{proof}

So under \eqref{eq: exponential_moments} with $\lambda \in [0,\alpha/2)$, we return to \eqref{eq:uppertail_after_holder} and find for some $\newconstant\label{D_{11}}>0$,
\begin{align}
\sum_{k=1}^\infty Ent(V_k^2) &\leq \lambda^2 \frac{\oldconstant{D_{11}} C_\lambda}{\#B_m} \sum_{z \in B_m} \left[ Ent(e^{2\lambda F_m}) + \mathbb{E}e^{2\lambda F_m} \right] \nonumber\\
&= \lambda^2 \oldconstant{D_{11}}C_\lambda \left[ Ent(e^{2\lambda F_m}) + \mathbb{E}e^{2\lambda F_m}\right]  \text{ for all } x \in \mathbb{Z}^d\ . \label{eq: positive}
\end{align}

\subsection{The case $\lambda \leq 0$}
When $\lambda \leq 0,$ the problem is again to bound above the term
\begin{equation}
\label{eq:expectbrack}
\frac{1}{\# B_m}\sum_{z \in B_m} \mathbb{E} \left[ e^{2\lambda F_m} Y_{z,x} \right]\ .
\end{equation}
We will break this up differently from before, now using a variant of the idea from \cite{DK}. Let $\newconstant\label{D_{12}}>0$ be arbitrary (to be fixed later, independent of $x$). Then
\begin{align}
\mathbb{E}e^{2\lambda F_m} Y_{z,x} &\leq \oldconstant{D_{12}} \mathbb{E} \left[ e^{2\lambda F_m} T_z \right]+ \mathbb{E} \left[ e^{2\lambda F_m} \, Y_{z,x}\mathbf{1}_{\{Y_{z,x}> \oldconstant{D_{12}} T_z\}} \right]\nonumber\\
\label{eq:introducingZ}
 &\leq  \oldconstant{D_{12}} \mathbb{E} \left[ e^{2\lambda F_m} \right]\mathbb{E} T_z+ \mathbb{E} \left[ e^{2\lambda F_m} Z_{z,x} \right],
\end{align}
where we have used the Harris-FKG inequality on the first term (since $\lambda \leq 0$, $e^{2\lambda F_m}$ is a decreasing function of $(t_e)$ whereas $T_z$ is increasing) and have defined the new variable 
\[ 
Z_{z,x}:= Y_{z,x} \mathbf{1}_{\{Y_{z,x}> \oldconstant{D_{12}} T_z\}}\ . 
\]

We will bound $\mathbb{P}(Z_{z,x}\geq n)$ in what follows. Analogously to the proof of Proposition \ref{prop: kubota}, define, for $\newconstant\label{fixie}>0$,
\[A'_n:= \left\{ \exists \text{ a self-avoiding } \gamma \text{ from } z \text{ to } z+x \text{ with } N(\gamma) \geq n \text{ but } T(\gamma) < \oldconstant{fixie} N(\gamma)\right\}. \]
Our first task is to control $\mathbb{P}(A_n')$.
\begin{lma}\label{lem: hanson_mansion}
There exist $\oldconstant{fixie},\newconstant\label{D_{13}}>0$ such that $\mathbb{P}(A_n') \leq e^{-\oldconstant{D_{13}} n}$ for all $n\geq 1$.
\end{lma}
\begin{proof}
By translation invariance we can consider $z=0$. The content of Proposition \ref{prop:latticeanbound} is that there exist constants $\newconstant\label{D_{15}},\newconstant\label{D_{16}} >0$ such that
\[
\mathbb{P}\left(\exists \text{ a self-avoiding }\gamma \text{ from } 0 \text{ with } \#\gamma = n \text{ such that } N(\gamma) > \oldconstant{D_{15}} n\right) \leq e^{-\oldconstant{D_{16}} n},~n \geq 1\ .
\]
Summing this over $n$ gives $\newconstant\label{D_{17}} < \infty$ such that
\begin{equation}
\label{eq:lengthbd2}
\mathbb{P}\left(\exists \text{ a self-avoiding }\gamma \text{ from } 0 \text{ such that } \#\gamma > n \text{ and }N(\gamma) > \oldconstant{D_{15}} \#\gamma\right) \leq e^{-\oldconstant{D_{17}} n}.
\end{equation}
Further, given $\newconstant\label{D_{18}}>0$,
\[
\mathbb{P}\left( \exists \text{ a self-avoiding }\gamma \text{ from } 0 \text{ with } \# \gamma \leq \oldconstant{D_{18}} n \text{ but } N(\gamma) \geq n\right) \leq \sum_{k=1}^{\oldconstant{D_{18}} n} \mathbb{P}(N_k \geq n)\ .
\]
If we choose $\beta$ in Proposition~\ref{prop:latticeanbound} for $\kappa =2$, then for some $\newconstant\label{D_{19}}, \newconstant\label{D_{20}}>0$, this is bounded by
\[
e^{-\beta n} \sum_{k=1}^{\oldconstant{D_{18}} n} \mathbb{E}e^{\beta N_k} \leq e^{-\beta n} \sum_{k=1}^{\oldconstant{D_{18}} n} e^{2k} = e^{-\beta n} \frac{e^{2(\oldconstant{D_{18}} n +1)}- e^2}{e^2 - 1} \leq \oldconstant{D_{19}} e^{-\oldconstant{D_{20}} n}\ ,
\]
if $\oldconstant{D_{18}}$ is small enough. Combining this with \eqref{eq:lengthbd2},
\begin{align*}
\mathbb{P}(A_n') &\leq \oldconstant{D_{19}} e^{-\oldconstant{D_{20}} n} + e^{-\oldconstant{D_{17}} \lfloor \oldconstant{D_{18}} n \rfloor} \\
&+ \mathbb{P}(\exists \text{ a self-avoiding } \gamma \text{ from } 0 \text{ with } \#\gamma > \oldconstant{D_{18}} n \text{ but } T(\gamma) < \oldconstant{fixie} \oldconstant{D_{15}} \#\gamma)\ .
\end{align*}
By \eqref{eq: kesten_length}, for small $\oldconstant{fixie}$, the last probability is bounded by $e^{-\newconstant n}$. Therefore $\mathbb{P}(A_n') \leq e^{-\newconstant n}$.
\end{proof}

From Lemma~\ref{lem: hanson_mansion}, we can decompose
\begin{align*}
\mathbb{P}\left(Z_{z,x} \geq n\right) &\leq \mathbb{P}\left(Z_{z,x} \geq n, \, (A_n')^c\right) + \mathbb{P}(A_n')\\
&\leq \mathbb{P}\left(Z_{z,x} \geq n, \, (A_n')^c\right) + e^{-\oldconstant{D_{13}} n}.
\end{align*}
Consider some outcome in $(A_n')^c$ such that $Z_{z,x} \geq n > 0$. For this outcome, we must have
\[ 
\oldconstant{D_{12}} T_z < Y_{z,x} \leq \frac{1}{\oldconstant{fixie}} T_z\ ,
\]
a contradiction for $\oldconstant{D_{12}} > \oldconstant{fixie}^{-1}$. This implies that independent of $x, z$ and $n,$ there exists $\oldconstant{D_{12}}$ such that
\[ \mathbb{P}\left(Z_{z,x} \geq n\right) \leq e^{-\oldconstant{D_{13}} n}\]
and, in particular, 
\[
\sup_{x \neq 0} \sup_{z \in B_m} \mathbb{E}e^{\delta Z_{z,x}} < \infty \text{ for } \delta = \oldconstant{D_{13}}/2\ .
\]

Now, to bound the second term of \eqref{eq:introducingZ} we apply Proposition~\ref{prop: variation} using our bound on $Z_{z,x}$. Namely, we obtain for some $\newconstant\label{D_{22}}$
\begin{align*}
\mathbb{E}e^{2\lambda F_m}Z_{z,x} &\leq \delta^{-1} \left[ Ent(e^{2 \lambda F_m}) + \mathbb{E}e^{2\lambda F_m} \log \mathbb{E}e^{\delta Z_{z,x}} \right] \\
&\leq \oldconstant{D_{22}} \left[ Ent(e^{2\lambda F_m}) + \mathbb{E}e^{2\lambda F_m} \right]\ ,
\end{align*}
implying
\[
\eqref{eq:expectbrack} \leq \newconstant \left[ Ent(e^{2\lambda F_m}) + (1+ \mathbb{E}T(0,x)) \mathbb{E}e^{2\lambda F_m} \right]\ .
\]

So we conclude that if $\lambda\leq 0$ and we assume \eqref{eq: two_moments}, then for some $\newconstant\label{D_{23}}$,
\begin{equation}\label{eq: negative}
\sum_{k=1}^\infty Ent(V_k^2) \leq \lambda^2 \oldconstant{D_{23}}C_{\lambda} \left[ Ent(e^{2\lambda F_m}) + (1+ \mathbb{E}T(0,x)) \mathbb{E}e^{2\lambda F_m} \right] \text{ for all } x \in \mathbb{Z}^d\ .
\end{equation}

\section{Proof of Theorem~\ref{thm: F_m_variance_bound}}
First we must complete the upper bound for $\sum_{k=1}^\infty Ent(V_k^2)$. What we have shown so far is (Theorem~\ref{thm: next_step}) that under \eqref{eq: exponential_moments} with $\lambda \in [0,\alpha/2)$ or \eqref{eq: two_moments} with $\lambda \leq 0$, setting $C_\lambda$ as in Theorem~\ref{thm: entropy_bound},
\[
\sum_{k=1}^\infty Ent(V_k^2) \leq \lambda^2 \oldconstant{D_6}C_\lambda \left[ Ent(e^{2\lambda F_m}) + (1+\mathbb{E}F_m)\mathbb{E}e^{2\lambda F_m} \right] \text{ for all } x \in \mathbb{Z}^d\ .
\]
This is close to the bound we would like, except there is an entropy term on the right. To bound this in terms of the moment generating function, we must use some techniques from Boucheron-Lugosi-Massart, similarly to what was done in Bena\"im-Rossignol (below (15) in \cite[Corollary~4.3]{BR}). Because these arguments lead us a bit astray, we place them in the appendix. By Theorem~\ref{thm: BLM_exponential} under \eqref{eq: exponential_moments}, we can transform the upper bound into, for some $\newconstant\label{D_{25}},\newconstant\label{D_{27}}>0$,
\[
\sum_{k=1}^\infty Ent(V_k^2) \leq \lambda^2 \oldconstant{D_{27}}C_\lambda \left(\|x\|_1 + 1 + \mathbb{E}F_m\right) \mathbb{E}e^{2\lambda F_m} \text{ for } x \in \mathbb{Z}^d \text{ and } 0 \leq \lambda \leq \oldconstant{D_{25}}\ .
\]
Using $\mathbb{E}F_m \leq \newconstant \|x\|_1$, we obtain
\[
\sum_{k=1}^\infty Ent(V_k^2) \leq \lambda^2 \newconstant C_\lambda \|x\|_1 \mathbb{E}e^{2\lambda F_m} \text{ for } x \in \mathbb{Z}^d \text{ and } 0 \leq \lambda \leq \oldconstant{D_{25}} \text{ under } \eqref{eq: exponential_moments}\ .
\]

On the other hand, when we assume \eqref{eq: two_moments}, Theorem~\ref{thm: BLM_two} gives the upper bound (with $\newconstant\label{D_{30}}$ from that theorem)
\[
\sum_{k=1}^\infty Ent(V_k^2) \leq \lambda^2\newconstant C_\lambda (1+\|x\|_1 + \mathbb{E}F_m)\mathbb{E}e^{2\lambda F_m} \text{ for } -\oldconstant{D_{30}}/2\leq \lambda \leq 0\ ,
\]
which again implies
\[
\sum_{k=1}^\infty Ent(V_k^2) \leq \lambda^2 \newconstant C_\lambda \|x\|_1 \mathbb{E}e^{2\lambda F_m} \text{ for } -\oldconstant{D_{30}}/2 \leq \lambda \leq 0 \text{ under } \eqref{eq: two_moments}\ .
\]
If we further restrict the range of $\lambda$ we can bound $C_\lambda$ using assumptions \eqref{eq: two_moments} and \eqref{eq: exponential_moments} and find for some $\newconstant\label{D_{32}}>0$,
\begin{equation}\label{eq: final_entropy_bound}
\sum_{k=1}^\infty Ent(V_k^2) \leq \lambda^2 \oldconstant{D_{32}}\|x\|_1 \mathbb{E}e^{2\lambda F_m}, x \in \mathbb{Z}^d\ ,
\end{equation}
where $-\newconstant\label{D_{33}} \leq \lambda \leq 0$ under \eqref{eq: two_moments} and $0 \leq \lambda \leq \oldconstant{D_{33}}$ under \eqref{eq: exponential_moments}.

We can finally place this bound back in the Falik-Samorodnitsky inequality \eqref{eq: BR_lower_bound} along with the bound on influences from Proposition~\ref{prop: V_k_bound}. We then obtain
\begin{equation}\label{eq: almost_done}
\Var e^{\lambda F_m} \leq \left[ \log \frac{\Var e^{\lambda F_m}}{\oldconstant{C_99} \lambda^2 \|x\|_1^{\frac{2+\zeta(1-d)}{2}} \mathbb{E}e^{2\lambda F_m}} \right]^{-1} \lambda^2 \oldconstant{D_{32}}\|x\|_1 \mathbb{E}e^{2\lambda F_m},~x \in \mathbb{Z}^d\ .
\end{equation}
Again, this holds for $-\newconstant\label{D_{40}}\leq \lambda \leq 0$ under \eqref{eq: two_moments} and $0 \leq \lambda \leq \oldconstant{D_{40}}$ under \eqref{eq: exponential_moments}. (Here we have used that $\oldconstant{D_{40}}$ can be slightly lowered to ensure that the term $\mathbb{E}((1+e^{\lambda t_e})t_e)^2$ is bounded by a constant under either assumption.)

From \eqref{eq: almost_done} we are almost done with the proof of Theorem~\ref{thm: F_m_variance_bound}. For any $d \geq 2$, $\frac{2+\zeta(1-d)}{2} \leq \frac{2-\zeta}{2}$, and so for every $\lambda$ either we have
\begin{equation}\label{eq: end_case_1}
\Var e^{\lambda F_m} \leq \oldconstant{C_99} \lambda^2 \|x\|_1^{\frac{4-\zeta}{4}} \mathbb{E}e^{2\lambda F_m}\ ,
\end{equation}
in which case we have inequalities \eqref{eq: variance_1} and \eqref{eq: variance_2} for $\|x\|_1 >1$ (with a possibly different constant, and replacing $\lambda$ with $2\lambda$), or the opposite inequality holds, in which case
\[
\frac{\Var e^{\lambda F_m}}{\oldconstant{C_99}\lambda^2 \|x\|_1^{\frac{2+\zeta(1-d)}{2}}\mathbb{E}e^{2\lambda F_m}} \geq \|x\|_1^{\frac{\zeta}{4}}\ .
\]
Therefore when \eqref{eq: end_case_1} fails,
\[
\Var e^{\lambda F_m} \leq \lambda^2 \oldconstant{D_{32}}\frac{\|x\|_1}{\frac{\zeta}{4} \log \|x\|_1} \mathbb{E}e^{2\lambda F_m}\ ,
\]
implying \eqref{eq: variance_1} and \eqref{eq: variance_2} again. This completes the proof of Theorem~\ref{thm: F_m_variance_bound}. 

To recap, Lemma~\ref{lem: iteration} shows that Theorem~\ref{thm: F_m_variance_bound} suffices to prove exponential concentration for $F_m$ (Theorem~\ref{thm: F_m_concentration}). Last, Theorem~\ref{thm: main_thm} follows from Theorem~\ref{thm: F_m_concentration} by the arguments in Section~\ref{sec: reduction}.

\appendix
\section{Preliminary entropy bounds}
In this appendix, we use ideas from the entropy method and from \cite{DK} to show bounds for the entropy of $e^{\lambda F_m}$. At the end we explain how these give a simple proof of Talagrand's inequality.

\subsection{Log Sobolev inequality}
We will use the ``symmetrized log Sobolev inequality'' of Boucheron-Lugosi-Massart \cite[Theorem~6.15]{BLM}.

\begin{thm}
Let $X$ be a random variable and let $X'$ be an independent copy. Then for $\lambda \in \mathbb{R}$,
\begin{equation}\label{eq: first_one}
Ent~e^{\lambda X} \leq \mathbb{E}\left[ e^{\lambda X} q(\lambda (X'-X)_+) \right]
\end{equation}
where $q(x) = x(e^x-1)$.
\end{thm}

\subsection{Application to $F_m$}

\subsubsection{Positive exponential}

\begin{thm}\label{thm: BLM_exponential}
Assuming \eqref{eq: exponential_moments}, there exist $\newconstant\label{C_55}, \newconstant{\label{D_{26}}}>0$ such that
\[
Ent~e^{\lambda F_m} \leq \oldconstant{C_55} \lambda^2 \|x\|_1 \mathbb{E}e^{\lambda F_m} \text{ for all } x \in \mathbb{Z}^d \text{ and } \lambda \in [0,\oldconstant{D_{26}}]\ .
\]
\end{thm}



\begin{proof}
By tensorization of entropy (Proposition~\ref{prop: tensor}),
\[
Ent~e^{\lambda F_m} \leq \sum_{k=1}^\infty \mathbb{E} Ent_{e_k}~e^{\lambda F_m}, \text{ for } \lambda \in [0,\alpha/2).
\]
Introduce $F_m^{(k)}$ as the variable $F_m$ evaluated at the configuration in which $t_{e_k}$ is replaced by an independent copy $t_{e_k}'$. Then we can apply \eqref{eq: first_one} conditionally:
\[
Ent~e^{\lambda F_m} \leq \sum_{k=1}^\infty \mathbb{E} \mathbb{E}_{e_k} e^{\lambda F_m} q(\lambda(F_m^{(k)}-F_m)_+)\leq \sum_{k=1}^\infty \mathbb{E} \mathbb{E}_{e_k} e^{\lambda F_m} q\left( \frac{1}{\#B_m} \sum_{z \in B_m} \lambda(T_z^{(k)}-T_z)_+ \right)\ .
\]
Convexity of $q$ on $[0,\infty)$ gives the upper bound
\[
\frac{1}{\#B_m} \sum_{z \in B_m} \sum_{k=1}^\infty \mathbb{E} \mathbb{E}_{e_k} e^{\lambda F_m} q\left( \lambda (T_z^{(k)}-T_z)_+ \right)\ .
\]
Lemma~\ref{lem: DHS_lem} implies that $(T_z^{(k)}-T_z)_+ \leq t_{e_k}' \mathbf{1}_{\{e_k \in Geo(z,z+x)\}}$ so we get the bound
\[
\frac{1}{\#B_m} \sum_{z \in B_m} \sum_{k=1}^\infty \mathbb{E}\mathbb{E}_{e_k} e^{\lambda F_m} q(\lambda t_{e_k}') \mathbf{1}_{\{e_k \in Geo(z,z+x)\}}\ .
\]
Integrate $t_{e_k}'$ first and bring the sum inside the integral for
\[
\frac{\mathbb{E}q(\lambda t_e)}{\#B_m} \sum_{z \in B_m} \mathbb{E} \left[ e^{\lambda F_m} \# Geo(z,z+x) \right]\ .
\]
Note that under \eqref{eq: exponential_moments}, $\mathbb{E}q(\lambda t_e) < \infty$ for $\lambda \in [0,2\alpha)$.

To deal with this product, use Proposition~\ref{prop: variation} with $X=pe^{\lambda F_m}$ and $Y=p^{-1}\#Geo(z,z+x)$, where $p>0$ is a parameter, to obtain 
\[
Ent~e^{\lambda F_m} \leq \frac{\mathbb{E}q(\lambda t_e)}{\#B_m} \sum_{z \in B_m} \left[ pEnt~e^{\lambda F_m} + p\mathbb{E}e^{\lambda F_m} \log \mathbb{E}\exp \left( \frac{\#Geo(z,z+x)}{p} \right) \right]\ .
\]
Note that $\frac{\mathbb{E}q(\lambda t_e)}{\lambda^2} \to \mathbb{E}t_e^2$ as $\lambda \downarrow 0$, so for some $\lambda' \in [0,\alpha/2)$, whenever $\lambda \in [0,\lambda']$,
\[
Ent~e^{\lambda F_m} \leq \lambda^2 \frac{2\mathbb{E}t_e^2}{\#B_m} \sum_{z \in B_m} \left[ pEnt~e^{\lambda F_m} + p \mathbb{E}e^{\lambda F_m} \log \mathbb{E}\exp\left( \frac{\#Geo(z,z+x)}{p} \right) \right]\ .
\]

By translation invariance, the expression inside the sum does not depend on $z$. So our bound is
\[
Ent~e^{\lambda F_m} \leq 2\lambda^2 \mathbb{E}t_e^2 \left( pEnt~e^{\lambda F_m} + p \mathbb{E}e^{\lambda F_m} \log \mathbb{E}e^{\# Geo(0,x)/p} \right)\ .
\]
Choose $p = \alpha_1$ from \eqref{eq: exponential_length_bound} and set $\lambda'' \in [0,\lambda']$ such that if $\lambda \in [0,\lambda'']$ then $2\lambda^2 \mathbb{E}t_e^2 p \leq 1/2$. We obtain the bound for some $A_1>0$,
\[
\frac{1}{2} Ent~e^{\lambda F_m} + 2\lambda^2 \mathbb{E}t_e^2 A_1 \|x\|_1 \mathbb{E}e^{\lambda F_m}\ .
\]
So
\[
Ent~e^{\lambda F_m} \leq 4\lambda^2 \mathbb{E}t_e^2 A_1 \|x\|_1 \mathbb{E}e^{\lambda F_m}\ .
\]
\end{proof}

\subsubsection{Negative exponential}
The bound given below is similar to the one derived in \cite{DK} for $T$ instead of $F_m$.
\begin{thm}\label{thm: BLM_two}
Assume $\mathbb{E}t_e^2<\infty$. There exist $\newconstant\label{C_77}, \oldconstant{D_{30}}>0$ such that
\[
Ent~e^{\lambda F_m} \leq \oldconstant{C_77} \lambda^2 \|x\|_1 \mathbb{E}e^{\lambda F_m} \text{ for all } x \in \mathbb{Z}^d \text{ and } \lambda \in [-\oldconstant{D_{30}},0]\ .
\]
\end{thm}
\begin{proof}
Again we use \eqref{eq: first_one} with tensorization:
\[
Ent ~e^{\lambda F_m} \leq \sum_{k=1}^\infty \mathbb{E}\mathbb{E}_{e_k} e^{\lambda F_m} q(\lambda (F_m^{(k)} - F_m)_+)\ ,
\]
By the inequality $q(x) \leq x^2 \text{ for } x \leq 0$ we obtain the upper bound
\[
\lambda^2 \sum_{k=1}^\infty \mathbb{E}\mathbb{E}_{e_k} e^{\lambda F_m} ((F_m^{(k)}- F_m)_+)^2\ .
\]
By convexity of $x \mapsto (x_+)^2$, this is bounded by
\[
\lambda^2 \frac{1}{\# B_m} \sum_{z\in B_m} \sum_{k=1}^\infty \mathbb{E} \mathbb{E}_{e_k} e^{\lambda F_m} ((T_z^{(k)}-T_z)_+)^2
\]
and using Lemma~\ref{lem: DHS_lem}, by
\[
\lambda^2 \frac{1}{\# B_m} \sum_{z\in B_m} \sum_{k=1}^\infty \mathbb{E} \mathbb{E}_{e_k} e^{\lambda F_m} (t_{e_k}' \mathbf{1}_{\{e_k \in Geo(z,z+x)\}})^2\ .
\]
Note that $t_{e_k}'$ is independent of $e^{\lambda F_m}\mathbf{1}_{\{e_k \in Geo(z,z+x)\}}$. So we integrate over $t_{e_k}'$ first to get
\begin{equation}\label{eq: almost_yeah}
\lambda^2 \mathbb{E}t_e^2 \frac{1}{\#B_m} \sum_{z \in B_m} \mathbb{E} e^{\lambda F_m} \#Geo(z,z+x)\ .
\end{equation}

To complete the proof, take $a$ from Proposition~\ref{prop: kubota} and upper bound the expectation as
\begin{equation}\label{eq: yeah}
(1/a) \mathbb{E}e^{\lambda F_m} T_z + \mathbb{E}e^{\lambda F_m} \#Geo(z,z+x) \mathbf{1}_{\{a\#Geo(z,z+x) >T_z\}}\ .
\end{equation}
The variable $e^{\lambda F_m}$ is decreasing as a function of the edge-weights (since $\lambda \leq 0$) whereas $T_z$ is increasing. So apply the Harris-FKG inequality to the first term for an upper bound of
\[
(1/a) \mathbb{E}e^{\lambda F_m} \mathbb{E}T_z = (1/a) \mathbb{E} e^{\lambda F_m} \mathbb{E}T(0,x)\ .
\]
For the second term call $Y = \#Geo(z,z+x)\mathbf{1}_{\{a\# Geo(z,z+x) > T_z\}}$ and use Proposition~\ref{prop: variation}. Taking $\delta = \oldconstant{C_2}/2$ from Proposition~\ref{prop: kubota}, we have $\mathbb{E}e^{\delta Y} \leq \mathbb{E}e^{\delta Y_x} < \newconstant\label{C_97}$ for some $\oldconstant{C_97}$ and so
\[
\mathbb{E}e^{\lambda F_m} Y \leq \delta^{-1} Ent~e^{\lambda F_m} + \delta^{-1} \oldconstant{C_97}\mathbb{E}e^{\lambda F_m}\ .
\]
Returning to \eqref{eq: almost_yeah}, we obtain
\[
Ent~e^{\lambda F_m} \leq \lambda^2 \mathbb{E}t_e^2 \left[ (1/a) \mathbb{E}e^{\lambda F_m} \mathbb{E}T(0,x) + \delta^{-1}Ent~e^{\lambda F_m} + \delta^{-1} \oldconstant{C_97} \mathbb{E}e^{\lambda F_m}\right]\ .
\]
Now restrict to $\lambda \leq 0$ such that $-\left(2\mathbb{E}t_e^2 \delta^{-1}\right)^{-1/2} \leq \lambda$ to obtain
\[
(1/2)Ent~e^{\lambda F_m} \leq \lambda^2 \mathbb{E}t_e^2 \left[ (1/a) \mathbb{E}e^{\lambda F_m} \mathbb{E}T(0,x) + \delta^{-1}\oldconstant{C_97} \mathbb{E}e^{\lambda F_m}\right]\ .
\]
Last, we bound $\mathbb{E}T(0,x) \leq \newconstant \|x\|_1$ to get 
\[
Ent~e^{\lambda F_m} \leq \lambda^2 \newconstant \|x\|_1 \mathbb{E}e^{\lambda F_m} \text{ for } x \in \mathbb{Z}^d,~ -(2\mathbb{E}t_e^2 \delta^{-1})^{-1/2} \leq \lambda \leq 0\ .
\]
\end{proof}

\subsection{Gaussian concentration for $F_m$}\label{sec: talagrand}

As a result of the above entropy bounds, we have the following concentration inequality for $F_m$. 
\begin{cor}
Assuming \eqref{eq: exponential_moments}, there exist positive $\newconstant{\label{C_lasagna}}$ and $\newconstant{\label{C_pizza}}$ such that for all $x \in \mathbb{Z}^d$,
\[
\mathbb{P}(F_m - \mathbb{E}F_m \geq \lambda \sqrt{\|x\|_1}) \leq e^{-\oldconstant{C_lasagna}\lambda^2} \text{ for } \lambda \in [0,\oldconstant{C_pizza} \sqrt{\|x\|_1}]\ .
\]
Assuming $\mathbb{E}t_e^2<\infty$, there exists $\newconstant{\label{C_pizza1}}>0$ such that for all $x \in \mathbb{Z}^d$,
\[
\mathbb{P}(F_m - \mathbb{E}F_m \leq -\lambda \sqrt{\|x\|_1}) \leq e^{-\oldconstant{C_pizza1}\lambda^2} \text{ for } \lambda \geq 0\ .
\]
\end{cor}
\begin{proof}
This is a standard application of the Herbst argument. An entropy bound of the type $Ent~e^{\lambda X} \leq C\lambda^2 \mathbb{E}e^{\lambda X}$ is rewritten as $\frac{\text{d}}{\text{d}\lambda} \frac{\log \mathbb{E}e^{\lambda(X-\mathbb{E}X)}}{\lambda} \leq C$. By integrating and using Theorems~\ref{thm: BLM_exponential} and \ref{thm: BLM_two}, we obtain
\[
\text{under \eqref{eq: exponential_moments}, }\log \mathbb{E}e^{\lambda(F_m-\mathbb{E}F_m)} \leq \oldconstant{C_55} \lambda^2 \|x\|_1 \text{ for } \lambda \in [0,\oldconstant{D_{26}}] 
\]
and
\[
\text{under }\mathbb{E}t_e^2<\infty,~ \log \mathbb{E}e^{\lambda(F_m-\mathbb{E}F_m)} \leq \oldconstant{C_77} \lambda^2 \|x\|_1 \text{ for } \lambda \in [-\oldconstant{D_{30}},0]\ .
\]
Markov's inequality then implies the upper tail inequality. For the lower tail inequality, Markov only gives
\[
\mathbb{P}(F_m - \mathbb{E}F_m \leq -\lambda \sqrt{\|x\|_1}) \leq e^{-\oldconstant{C_pizza1}\lambda^2} \text{ for } \lambda \in [0,\newconstant \sqrt{\|x\|_1}]\ .
\]
However, noting that for $B = \sup_{x \neq 0} \mathbb{E}T(0,x) / \|x\|_1$, one has $\mathbb{P}(T(0,x) - \mathbb{E}T(0,x) \leq -B \|x\|_1) = 0$, we can decrease \oldconstant{C_pizza1} to deduce the lower tail inequality for all $\lambda \geq 0$.
\end{proof}

By bounding the error $|T(0,x) - F_m|$ as in Section~\ref{sec: reduction}, we obtain a simple proof of Talagrand's inequality \cite[Eq.~(1.15)]{talagrand}:
\begin{cor}
Assuming \eqref{eq: exponential_moments}, there exist positive $\newconstant{\label{C_pizza2}}$ and $\newconstant{\label{C_pizza3}}$ such that for all $x \in \mathbb{Z}^d$,
\[
\mathbb{P}(T(0,x) - \mathbb{E}T(0,x) \geq \lambda \sqrt{\|x\|_1}) \leq e^{-\oldconstant{C_pizza2}\lambda^2} \text{ for } \lambda \in [0,\oldconstant{C_pizza3} \sqrt{\|x\|_1}]\ .
\]
Assuming $\mathbb{E}t_e^2<\infty$, there exists $\newconstant{\label{C_pizza4}}>0$ such that for all $x \in \mathbb{Z}^d$,
\[
\mathbb{P}(T(0,x)- \mathbb{E}T(0,x) \leq -\lambda \sqrt{\|x\|_1}) \leq e^{-\oldconstant{C_pizza4}\lambda^2} \text{ for } \lambda \geq 0\ .
\]
\end{cor}

\bigskip
\noindent
{\bf Acknowledgements.} P. S. thanks IU Bloomington for accommodations and hospitality during a visit when work was done on this project. The authors warmly thank an anonymous referee for an earlier version of this manuscript for detailed comments and supplying the current statement and proof of Lemma~\ref{lma: uniformvar}.

\end{document}